\documentclass[11pt]{article}
\usepackage{xypic}
\usepackage{mathrsfs}
\usepackage{graphicx}
\usepackage{enumerate}
\usepackage{amsfonts,amsmath,amsthm, amssymb}
\usepackage{latexsym, euscript, epic, eepic, color}
\usepackage{hyperref}
\numberwithin{equation}{section}
\newtheorem{lemma}{Lemma}[section]
\newtheorem{theorem}{Theorem}[section]
\newtheorem{definition}{Definition}[section]
\newtheorem{proposition}{Proposition}[section]
\newtheorem{corollary}{Corollary}

\newtheorem{remark}{Remark}
\newtheorem{example}{Example}[section]
\newcommand{\ba}{\mathbf{a}}
\newcommand{\bb}{\mathbf{b}}

\newcommand{\bu}{\mathbf{u}}
\newcommand{\bv}{\mathbf{v}}

\begin{document}

\title{Open dynamical systems in topologically expansive decreasing Lorenz maps
\footnotetext {$\dag$ Corresponding author:  sunyun@whut.edu.cn}
}

\author{Yiming Ding$^1$, Wolfgang Steiner$^2$, Yun Sun$^{3,\dag}$, \\
\small \it $^1$School of Mathematics and Systems Science,\\
\small \it Wuhan University of Science and Technology, Wuhan 430081, China\\
\small \it $^2$IRIF, Universit\'e Paris Cit\'e, CNRS, F--75006 Paris, France \\
\small \it $^{3,\dag} $Hubei Key Laboratory of Mathematics $\&$ Intelligent Computation\\
\small \it and Interdisciplinary Sciences;\\
\small \it School of Mathematics and Statistics, Wuhan University of Technology,\\
\small \it Wuhan 430070, P.R. China}
\date{}
\maketitle
\begin{center}
\begin{minipage}{120mm}
{\small {\bf Abstract.}
We develop a parity-sensitive admissibility theory for topologically
expansive decreasing Lorenz maps. By embedding them into the standard
circle map $m_{-2}$, we realize every decreasing-admissible pair.
We then study the negative doubling map $T_{-2}$ with a hole $(a,b)$,
where $0<a\leq1/2\leq b<1$. Using the alternating lexicographic order,
we classify plateaux of the symbolic survivor spaces. Under weak
admissibility and strict shifted ordering, the boundary subshifts are
realized as kneading spaces of decreasing Lorenz maps. We also establish
entropy correspondences across countable itinerary discrepancies in
completed survivor systems. For fixed $a$, we prove that
$b\mapsto\dim_{\mathrm H}S_{-2}(a,b)$ is a devil's staircase, possibly constant.
Under the pullback-null condition~(P), we obtain the corresponding
survivor-entropy result for expansive decreasing Lorenz maps. We verify
this condition for the piecewise-linear family in
Example~\ref{ex:edlm-linear-pullback}.
}

\end{minipage}
\end{center}
\vskip0.5cm

\section{Introduction}
\subsection{Lorenz maps}
Let $0< c <1$. (Increasing) Lorenz maps are interval maps $f:[0,1] \to [0,1]$ such that
\begin{enumerate}[\upshape (i)]
\item
$f$ is strictly increasing on $[0,c)$ and $(c,1]$,
\item \label{i:i2}
 $f(c-)=1$ and $f(c+)=0$.
\end{enumerate}
If additionally
\begin{enumerate}[\upshape (i)]
\setcounter{enumi}{2}
\item
the set of preimages $C(f) = \bigcup_{n\ge0} f^{-n}(c)$ of $c$ is dense in $[0,1]$,
\end{enumerate}
then $f$ is called a \emph{topologically expansive (increasing) Lorenz map}.

Lorenz maps are one-dimensional maps with a single discontinuity.  They
arise as Poincar\'e return maps for flows on branched manifolds that model
strange attractors of Lorenz systems and play an important role in the study
of vector fields near homoclinic bifurcations; see
\cite{Morales2005Expanding,morales2006contracting,robinson2000nonsymmetric}.

\par For an increasing Lorenz map, label the branches by $0$ and $1$.
For $x\notin C(f)$, its itinerary $k_f(x)=\epsilon_0\epsilon_1\cdots$
records these labels along the orbit. At precritical points use the
one-sided limits $k_f(x-)$ and $k_f(x+)$, taken through points outside
$C(f)$. The kneading invariants are $k_-=k_f(c-)$ and $k_+=k_f(c+)$,
and the kneading space is $\Omega(f)=\{k_f(x\pm):x\in[0,1]\}$.
Write $\prec$ for the usual lexicographic order and $\sigma$ for the
left shift. Hubbard and Sparrow~\cite[Theorem~1]{Hubbard1990} proved
that the kneading invariants of an expansive Lorenz map satisfy
\begin{equation} \label{expanding}
\sigma (k_{+})\preceq \sigma^n(k_{+})\prec\sigma (k_{-}) \ \ {\rm and} \ \ \sigma (k_{+})\prec \sigma^n(k_{-})\preceq\sigma (k_{-}) \quad \mbox{for all}\ n \ge 0;
\end{equation}
conversely, given any two sequences $k_{-}$ and $k_{+}$ satisfying \eqref{expanding}, there exists an expansive Lorenz map $f$ with $(k_{-},k_{+})$ as its kneading invariant, and $f$ is unique up to conjugacy by a homeomorphism of $[0,1]$.
Consequently, we call the condition \eqref{expanding} the \emph{HS admissible condition}.
We also recall that, by \cite[Theorem~2]{Hubbard1990}, the kneading space $\Omega(f)$ of an expansive Lorenz map~$f$ is completely determined by the kneading invariants $(k_-,k_+)$ of~$f$, i.e.,
\[
\Omega(f)= \{\omega \in \{0,1\}^{\mathbb{N}} \,:\, \sigma(k_{+})\preceq  \sigma^{n}(\omega) \preceq \sigma(k_{-}),  \forall \ n \in \mathbb{N}_{0}\}.
\]

The HS admissibility condition has been used in many results on
topologically expansive Lorenz maps, including kneading theory
\cite{Glendinning1996,Glendinning1993}, renormalization
\cite{cholewa2024,Ding2011}, open dynamical systems
\cite{sun2025,sun2024}, matching
\cite{Bruin2017,sun2023}, and linearization
\cite{Cui2015,dingsun2026,Glendinning1990}.

The goal of this paper is to extend some of these results to decreasing
Lorenz maps. We use a decreasing version of the \emph{extended Lorenz
map} definition: we require $f(0)=1$ and $f(1)=0$, while the two
critical values may vary. This differs from fixing the extremal
values at the discontinuity, as in the usual increasing definition above
and its decreasing version $f(c-)=0$, $f(c+)=1$.
The endpoint conditions $f(0)=1$ and $f(1)=0$ agree with the class C
normalization of Glendinning and Hege~\cite[Section 2]{GlendinningHege2026}.

\begin{definition}\label{defofdecreasing}
{A~\emph{decreasing Lorenz map} is an interval map $f:I\to I$, where $I:=[0,1]$, with a critical point $c\in (0,1)$, such that}
\begin{enumerate}[\upshape (i)]
\item
$f$ is continuous and strictly decreasing on $[0,c)$ and on $(c,1]$;
\item \label{i:d2}
$f(0)=1$, $f(1)=0$, and the one-sided limits satisfy
$f(c+)>f(c-)$.
\end{enumerate}
If additionally
\begin{enumerate}[\upshape (i)]
\setcounter{enumi}{2}
\item
{the preimage set $C(f) = \bigcup_{n\ge0} f^{-n}(c)$ of $c$, with preimages taken up to the first visit to $c$, is dense in $[0,1]$,}
\end{enumerate}
then $f$ is called a \emph{topologically expansive decreasing Lorenz map}.
\end{definition}
{
Throughout, \(f\) is a map from all of \(I\) to \(I\). Its value
\(f(c)\in I\) may be chosen arbitrarily; a formula specifying only the
two branches is understood to be supplemented by such a value at \(c\).
The two formal sides \(c-\) and \(c+\) specify the one-sided limits
\(f(c-)\) and \(f(c+)\), which determine the kneading sequences
independently of the choice of \(f(c)\).
The notation \(f^{-0}(c)=\{c\}\), while for \(n\geq1\), \(f^{-n}(c)\)
uses branch iterates that do not hit \(c\) before the \(n\)-th step.
For survivor sets with \(c\) as a hole endpoint, iteration at \(c\)
uses the corresponding branch-limit value selected by the left or right
completion convention introduced after \eqref{suanfa2}.
}
We write $\mathrm{DLM}$ for the class of decreasing Lorenz maps and
$\mathrm{EDLM}$ for the topologically expansive subclass. The upward jump
in condition~(\ref{i:d2}) is forced by expansivity.
We now explain how a map with different endpoint values can be included
in this setting.
A map $g:[0,1]\to[0,1]$ with two continuous strictly decreasing branches,
an upward jump, and $g(0)<1$, $g(1)>0$ can be extended affinely on both
sides to a larger interval, so that the left endpoint is mapped to the
right endpoint and the right endpoint is mapped to the left endpoint.
Rescaling this interval to $[0,1]$ gives $f(0)=1$, $f(1)=0$, while
preserving the original dynamics, up to affine conjugacy, on an invariant
subinterval.
For the special linear class of decreasing Lorenz-type maps,
negative $\beta$-transformations provide important examples and have been studied in
\cite{frougny2009,huang2026,ito2009,wolfgang2011,wolfgang2013}.

 Let $f\in\mathrm{EDLM}$.  As in the increasing case, trajectories can be coded
by sequences in $\{1,2\}^{\mathbb N}$.  For
$\omega,\nu\in\{1,2\}^{\mathbb N}$, the alternating lexicographic order
$\prec$ is defined by $\omega\prec\nu$
if there exists an integer $k \ge 1$ such that $\omega_{j}=\nu_j$ for all $1 \le j < k$ and
\begin{equation}\label{eq:alter-order}
\begin{cases}
\omega_k < \nu_k & \text{if } k \text{ is odd}, \\
\omega_k > \nu_k & \text{if } k \text{ is even}.
\end{cases}\end{equation}
 This is the alternating order used for maps with two decreasing
branches; see Glendinning~\cite[Section~3]{Glendinning2014}.
{\bf Throughout the remainder of this paper, all comparisons of symbolic
sequences and words use the alternating lexicographic order.}
 The kneading sequence of a point $x$ is the sequence
$\tau_f(x)=\epsilon_0\epsilon_1\ldots\in\{1,2\}^{\mathbb N}$ defined by
\begin{equation}\label{suanfa}
    \epsilon_i=1\ \ \ \ \ \ {\rm if}\ \ \ f^i(x)<c \ \ \ \ {\rm and} \ \ \
\epsilon_i=2\ \ \ \ \ \ {\rm if}\ \ \ f^i(x)>c.
\end{equation}
This definition works for $x \notin C(f)=\bigcup_{n \ge 0}f^{-n}(c)$. If $x$ is
a preimage of $c$, $x$ has upper and lower kneading sequences
\begin{equation}\label{suanfa2}\tau_f(x+)=\lim_{y\downarrow x}\tau_f(y), \qquad
\tau_f(x-)=\lim_{y\uparrow x}\tau_f(y),\end{equation}
where $y$ ranges over points
of $[0,1]$ that are not preimages of $c$.  In particular,
$(k_-, k_+)=(\tau_f(c-), \tau_f(c+))$ are the \emph{ kneading invariants} of $f$.
The symbols $\tau_f(x-)$ and $\tau_f(x+)$ always denote these
one-sided limits. At points outside $C(f)$, they agree with the unique
itinerary $\tau_f(x)$; at $0$ and $1$, only the limit from inside $I$
is used. For a nonempty hole with $a=c<b$, the \emph{left
completion} means that survivor orbits use the value $f(c-)$ at every
visit to $c$, which is labelled by $1$. For $a<c=b$, the \emph{right
completion} uses $f(c+)$ at every visit to $c$, labelled by $2$.
Away from $c$, both completions agree with $f$.
The itinerary of a completed orbit is obtained by iterating this chosen
map; if the critical orbit returns to $c$, it need not equal the limiting
sequence $\tau_f(c-)$ or $\tau_f(c+)$, respectively. We continue to use
the one-sided limits as the boundary sequences of the hole. The possible
differences at critical and precritical itineraries are accounted for in
Proposition~\ref{proptwoholes}. When $a=b=c$, the hole is empty, the survivor
set is $I$, and the full kneading space with both one-sided sequences is
retained.

The kneading space of $f$, also called the decreasing Lorenz shift, is
$\Omega(f)=\{\tau_f(x\pm): x \in [0,1]\}.$ These itinerary and
kneading-space definitions are also used for general maps in
$\mathrm{DLM}$. Define
\begin{equation*}
\left \{ \begin{array}{ll}
\Omega(k_-,k_+):=\{\omega\in\{1,2\}^\mathbb{N}:\sigma^n(\omega)\succeq k_+ \ {\rm or} \ \sigma^n(\omega)\preceq k_- \ \forall n\in \mathbb{N}_{0} \},\\
\Omega[k_-,k_+):=\{\omega\in\{1,2\}^\mathbb{N}:\sigma^n(\omega)\succeq k_+ \ {\rm or} \ \sigma^n(\omega)\prec k_- \ \forall n\in \mathbb{N}_{0} \},\\
\Omega(k_-,k_+]:=\{\omega\in\{1,2\}^\mathbb{N}:\sigma^n(\omega)\succ k_+ \ {\rm or} \ \sigma^n(\omega)\preceq k_- \ \forall n\in \mathbb{N}_{0} \}.
\end{array}
\right.
\end{equation*}

\begin{theorem}\label{thm1}
If $f\in\mathrm{EDLM}$, then its kneading invariants satisfy
\[
k_-\in1\{1,2\}^{\mathbb N},\qquad
k_+\in2\{1,2\}^{\mathbb N},\qquad
\sigma(k_-)\prec\sigma(k_+),
\]
and, for every $j\geq0$,
\begin{equation}\label{eq:parity-admissibility}
\left\{
\begin{aligned}
\sigma^{2j}(k_-)&\notin(k_-,k_+],\qquad
\sigma^{2j}(k_+)\notin[k_-,k_+),\\
\sigma^{2j+1}(k_-)&\notin[k_-,k_+),\qquad
\sigma^{2j+1}(k_+)\notin(k_-,k_+]
\end{aligned}
\right.
\end{equation}
where all inequalities are taken in the alternating order.
Conversely, every pair satisfying the displayed conditions is the kneading
pair of a map $f\in\mathrm{EDLM}$. {Moreover, if \(f,g\in\mathrm{EDLM}\) have the same kneading
pair and critical points \(c_f,c_g\), then there is an increasing
homeomorphism \(h:[0,1]\to[0,1]\) with \(h(c_f)=c_g\) such that
\[
\begin{aligned}
h(f(x))&=g(h(x))\quad\text{for }x\in I,\ x\ne c_f,\\
h\bigl(f(c_f\pm)\bigr)&=g(c_g\pm).
\end{aligned}
\]
Thus the kneading pair determines the two monotone branches and their
one-sided limits up to an increasing change of coordinates. This
uniqueness statement makes no assertion about the independently chosen
values \(f(c_f)\) and \(g(c_g)\).}
\end{theorem}

\begin{theorem}\label{thm2}
Let $f\in\mathrm{EDLM}$ with kneading invariants $(k_-,k_+)$. Then
\[
\Omega(f)=\{\omega\in\{1,2\}^{\mathbb N}: \sigma^n(\omega)\notin (k_-,k_+)
\ \text{for all } n\ge 0\}.
\]
Conversely, given any $\omega\in \Omega(f)$, it is
the upper or lower kneading sequence of some point $x\in[0,1]$.
Two distinct sequences in $\Omega(f)$ correspond to the same point exactly
when they are of the form $w\,k_-$ and $w\,k_+$ for some finite word $w$.
These sequences are precisely the upper and lower kneading sequences of a
point whose orbit eventually reaches~$c$.
\end{theorem}

\begin{corollary}\label{co1}
For each chosen corresponding fixed point and the normalization sending it
to \(0\), the orientation-preserving conjugacy in the circle realization
is unique.
\end{corollary}
\begin{proof}\rm
This is the normalized uniqueness in Proposition~\ref{po1}, applied to
the corresponding fixed point.
\end{proof}

\subsection{Open dynamical systems}
\par Dynamical systems with holes study points that do not fall into a
prescribed set under iteration.  Such systems were introduced by Pianigiani
and Yorke \cite{Pianigiani1979}.  They raise questions about both arithmetic
properties and dynamical behaviour (see, for example,
\cite{glendinning2015,Sidorov2003,Sidorov2014}).  In the general setting,
one considers a discrete dynamical system $(I,f)$, where $I=[0,1]$ and
$f:I\to I$ is continuous or piecewise continuous with positive topological
entropy.  Let $H\subset I$ be a connected subinterval, called the hole.  We
focus on the survivor set
$$ S_f(H)=\{x\in I: f^{n}(x)\notin H \  \forall n\geq 0\}=I\setminus\bigcup_{n=0}^{\infty}f^{-n}(H).
$$
The survivor set varies with both the hole and the map.  Many studies
investigate its Hausdorff dimension and the topological entropy
$h_{\mathrm{top}}(f|_{S_f(H)})$.  For a hole $H=(0,t)$, write
$\eta_f(t)=\dim_{\mathrm H}(S_f(0,t))$.  Urba\'nski
\cite{Urbanski1986,Urbanski1987} proved that this is a devil's staircase for
the doubling map $T_2$.  Kalle et al.~\cite{kalle2020} extended the result
to $T_\beta$ with $\beta\in(1,2]$, and Langeveld and Samuel
\cite{Langeveld2023} treated intermediate $\beta$-transformations.  Sun et
al.~\cite{sun2024} considered an expansive increasing Lorenz map with a hole
$(a,b)$ containing the discontinuity $c$ and proved that, if the map has an
ergodic absolutely continuous invariant probability measure, then
$b\mapsto h_{\mathrm{top}}(f|_{S_f(a,b)})$ is a devil's staircase for fixed
$a\in[0,c]$.
   \par All the above-mentioned results concern  increasing  Lorenz maps
or piecewise expanding maps of positive slope. Two natural questions arise:
\begin{enumerate}
\item can we extend the devil's staircase property to the decreasing cases?
\item what is the relationship between $f\in\mathrm{EDLM}$ and negative doubling maps with a hole?
\end{enumerate}

 In this paper, we start with the interval negative-doubling map
\[
T_{-2}(x):=
\begin{cases}
-2x+1,&x\in[0,1/2),\\
-2x+2,&x\in(1/2,1].
\end{cases}
\]
 We consider a hole $H=(a,b)$ with
$0<a\leq1/2\leq b<1$ and the associated survivor set $S_{-2}(a,b)$.
The endpoint cases \(a=0\) and \(b=1\) are excluded, since their survivor
sets have Hausdorff dimension zero. Indeed, when \(a=0\), every survivor
either eventually enters the endpoint cycle \(\{0,1\}\), or stays on the
right branch forever and is the fixed point \(2/3\). The case \(b=1\)
is symmetric, with fixed point \(1/3\). Thus these survivor sets are at
most countable. For the holes \((0,1/2)\) and \((1/2,1)\), the completed
critical endpoint and its surviving preimages give countably infinite
survivor sets.

Let $\ba$ be the lower kneading sequence of $a$ and
$\bb=\tau_{-2}(b+)$ the upper kneading sequence of $b$.
For the parameter definitions below, right endpoints range over
$[1/2,1)$.
Associate with $S_{-2}(a,b)$ the subshift determined by its boundary sequences:
\[
\Omega(\ba,\bb) = \{ w \in \{1,2\}^{\mathbb N} \,:\, \sigma^n w \notin (\ba,\bb) \ \mbox{for all}\ n \ge 0\}.
\]
For nonempty critical-endpoint holes, the compact itinerary subshift of
the completed survivor system may differ from this space at countably
many precritical itineraries. Their entropies agree, as established in
Proposition~\ref{proptwoholes}.
Under the extended definition of decreasing Lorenz maps used in this paper,
Theorem~\ref{thm2} identifies the kneading space with the symbolic survivor
space of $T_{-2}$ for the corresponding hole, via the negative-doubling
address map and the one-sided endpoint conventions below. If an additional
hole contains $c$, the forbidden symbolic interval enlarges, giving the two
nested holes in Proposition~\ref{proptwoholes}; this explains the use of
the extended definition.

For a general decreasing Lorenz map, the boundary sequences satisfy the
gap-avoidance condition
\[
k_-,k_+\in\Omega(k_-,k_+),
\]
which we call \emph{weak-admissibility}. This is the compatibility condition
appearing in earlier work on decreasing Lorenz maps
\cite{Glendinning2014,GlendinningHege2026}. The stronger parity conditions
in Theorem~\ref{thm1} concern the expansive subclass $\mathrm{EDLM}$; no
uniqueness or expansivity is asserted for the general DLM realization.
The following result gives the precise DLM correspondence used below.
\begin{theorem}\label{thm:dlm-hole}
Let $T_{-2}$ have a hole $H=(a,b)$ with $0<a\leq1/2\leq b<1$ and boundary
sequences $\ba,\bb$. If
\(\ba,\bb\in\Omega(\ba,\bb)\) and
\(\sigma(\ba)\prec\sigma(\bb)\),
then there exists $g\in\mathrm{DLM}$ such that
\(\Omega(g)=\Omega(\ba,\bb)\).

Conversely, if $g\in\mathrm{DLM}$ has kneading pair $(k_-,k_+)$, then
\(k_-,k_+\in\Omega(k_-,k_+)\) and
\(\Omega(g)=\Omega(k_-,k_+)\). Set
\(a:=\pi_{-2}(k_-),\ b:=\pi_{-2}(k_+)\). If
\(0<a<1/2<b<1\), then
\(\tau_{-2}(a-)=k_-\), \(\tau_{-2}(b+)=k_+\), and
\(\Omega(g)=\Omega(\tau_{-2}(a-),\tau_{-2}(b+))\).
The same conclusion holds in the degenerate case
\(a=b=1/2\) with
\((k_-,k_+)=(1(12)^\infty,2(21)^\infty)\), for which the hole is empty.
\end{theorem}

Glendinning~\cite{GlendinningAnharmonic1993,Glendinning2014} developed
admissibility and renormalization theory for maps with decreasing branches.
Glendinning and Hege~\cite[Theorems~10 and~13]{GlendinningHege2026}
realize compatible kneading pairs in complete families of plateau maps and
describe the boundary of positive entropy, equivalently the boundary of
positive survivor-set dimension for the associated open maps. In the
decreasing case, their compatibility condition is the gap-avoidance
condition $\sigma^n(k_\pm)\notin(k_-,k_+)$ for all $n\geq0$, corresponding
to our weak-admissibility.

Keeping \(a\) fixed, \(\Omega(\ba,\bb)\) may remain unchanged as \(b\)
varies.  For fixed \(\ba\), set
\[
P(\bb):=\bigl\{\epsilon\in 2\{1,2\}^{\mathbb{N}}:
\Omega(\ba,\epsilon)=\Omega(\ba,\bb)\bigr\}.
\]
Thus \(P(\bb)\) is the plateau of right boundaries with the same
survivor space as \(\bb\). The following theorem gives its explicit structure and characterizes
when it is a singleton.

\begin{theorem}\label{plateauxin}
Fix $\ba$ and suppose that $(\ba,\bb)$ is weak-admissible.
Define
\[
n:=\min\{2j:j\geq1,\ \sigma^{2j}(\bb)=\bb\},
\qquad
m:=\min\{2j+1:j\geq0,\ \sigma^{2j+1}(\bb)=\ba\},
\]
with  $\min\varnothing=+\infty$, and put
$k:=\min\{n,m\}$.
If $k<+\infty$, put $w=\bb|_k$ and define
\[
L=
\begin{cases}
w\ba,&k\text{ even},\\
w^\infty,&k\text{ odd}.
\end{cases}
\]
Then $L\in\Omega(\ba,L)$ and $P(\bb)=(L,\bb]$.
If $k=+\infty$, then $P(\bb)=\{\bb\}$.
\end{theorem}

  This plateau structure provides the symbolic basis for the
Hausdorff-dimension result stated next.

In this paper, the \emph{devil's staircase property} means that a function
is continuous and monotone, and is locally constant outside a set of
Lebesgue measure zero. We allow the degenerate case of a constant function.

\begin{theorem}\label{theorem2}
Let $T_{-2}$ have a hole $(a,b)$ with $0<a\leq1/2\leq b<1$. Then, for fixed $a$,
the Hausdorff dimension function
\[
\eta(a,\cdot):[1/2,1)\longrightarrow[0,1],\qquad
b\longmapsto \dim_{\mathrm H}S_{-2}(a,b)
\]
is a devil's staircase in the above sense.
\end{theorem}
It is clear that $\eta(a,\cdot)$ is nonincreasing. If $b$ is fixed, the
dimension function $\eta(\cdot,b)$ is nondecreasing and has the analogous
devil's staircase property.
The following corollary gives the corresponding EDLM version
of the relation between decreasing Lorenz maps and negative-doubling systems
described in
Theorem~\ref{thm:dlm-hole}.

\begin{corollary}\label{coco2}
Let $f\in\mathrm{EDLM}$ and let $H=(a,b)$ be a hole with $0<a\leq c\leq b<1$ ($c$ is the discontinuity). Put
$
\ba:=\tau_f(a-)$,  $\bb:=\tau_f(b+).
$
Assume that
$
\ba,\bb\in\Omega(\ba,\bb)$ and
$\sigma(\ba)\prec\sigma(\bb).
$
Then there exists $g\in\mathrm{DLM}$ such that
$
\Omega(g)=\Omega(\ba,\bb).
$
Conversely, let $g\in\mathrm{DLM}$ have kneading pair $(k_-,k_+)$
whose negative-doubling addresses satisfy
\[
0<a:=\pi_{-2}(k_-)<1/2<b:=\pi_{-2}(k_+)<1.
\]
Then the map $T_{-2}\in\mathrm{EDLM}$ with the hole $H=(a,b)$ satisfies
$
\tau_{-2}(a-)=k_-,$ $\tau_{-2}(b+)=k_+,
$
and
$
\Omega\bigl(\tau_{-2}(a-),\tau_{-2}(b+)\bigr)
=\Omega(k_-,k_+)=\Omega(g).
$
The same conclusion also holds in the canonical degenerate
case \(a=b=1/2\) with
\((k_-,k_+)=(1(12)^\infty,2(21)^\infty)\), where the hole is empty.
\end{corollary}

 The parity-admissible kneading classification gives the
symbolic correspondence between an $f\in\mathrm{EDLM}$ with a hole and
a negative-doubling system with the corresponding boundary sequences. A
Lebesgue-measure statement for the $f$-parameter requires the additional
pullback condition in Corollary~\ref{coco3}.

\begin{corollary}\label{coco3}
Let $f\in\mathrm{EDLM}$ have a hole $(a,b)$ with $0<a\leq c\leq b<1$, where $c$
is its discontinuity, and fix the left endpoint $a$. The entropy of the
survivor system is symbolically the entropy of a negative doubling system
with the corresponding boundary sequences.
 If the pullback-null condition~\eqref{eq:pullback-null} in
Proposition~\ref{prop:edlm-null} holds, then
$b\mapsto h_{\mathrm{top}}(S_f(a,b))$ is a devil's staircase in the same sense,
allowing the constant case.
\end{corollary}

Section~\ref{sese2} collects notation and entropy and dimension results.
Section~\ref{secsec3} proves the EDLM classification and kneading-space
description in Theorems~\ref{thm1} and~\ref{thm2}.
Section~\ref{secsec4} proves Theorems~\ref{thm:dlm-hole},
\ref{plateauxin}, and~\ref{theorem2}, covering DLM realization, survivor
plateaux, and dimension. Section~\ref{secfinal} gives the EDLM entropy
correspondence and conditional Lebesgue-null result.

\section{Preliminaries}\label{sese2}
\subsection{Notations}
For the negative doubling map defined in the introduction, the completed
critical values are
\[
T_{-2}(c-):=\lim_{x\uparrow c}T_{-2}(x)=0,\qquad
T_{-2}(c+):=\lim_{x\downarrow c}T_{-2}(x)=1.
\]
By equation~\eqref{suanfa2}, the kneading invariants of \(T_{-2}\) are
$(k_-,k_+)=(1(12)^\infty,2(21)^\infty)$.
For $u=u_1u_2\cdots\in\{1,2\}^{\mathbb{N}}$, define the  address map by
$$
\pi_{-2}(u):=-\sum_{n=1}^{\infty}\frac{u_n}{(-2)^n}\in[0,1].
$$
The alphabet $\{1,2\}$ makes this address formula direct; with labels
$\{0,1\}$, one would replace $u_n$ by $u_n+1$.

 For a symbolic sequence
$\omega=\omega_1\omega_2\cdots\in\{1,2\}^{\mathbb N}$, we use the
corresponding circle projection
\[
[\omega]=\pi_{-2}(\omega)\pmod 1\in S^1.
\]
Thus \(\pi_{-2}\) is the real-valued interval address map, while
the brackets \([\omega]\) denote the circle address of the sequence
$\omega$.
Let
\[
C_{-2}:=\bigcup_{n\geq0}T_{-2}^{-n}(1/2)
\]
be the critical and precritical set.  Here \(\tau_{-2}(x\pm)\) denotes
the one-sided itinerary with respect to the partition at \(1/2\).  If
\(x\notin C_{-2}\), the itinerary of \(x\) is unique.  If \(x\in C_{-2}\), we use the two
one-sided itineraries \(\tau_{-2}(x-)\) and \(\tau_{-2}(x+)\); they are
the only possible ambiguity and satisfy
\[
\pi_{-2}\bigl(\tau_{-2}(x-)\bigr)
=\pi_{-2}\bigl(\tau_{-2}(x+)\bigr)=x.
\]
The two sequences at a precritical point remain distinct symbolic
boundaries despite having the same interval address.

We equip the space $\{1,2\}^\mathbb{N}$ of infinite sequences with the topology induced by the usual metric $d \colon \{1,2\}^\mathbb{N} \times \{1,2\}^\mathbb{N} \to \mathbb{R}$ which is given by
\begin{align}
d(\omega, \nu) :=
\begin{cases}
0 & \text{if} \; \omega = \nu,\\
2^{- \lvert\omega \wedge \nu\rvert + 1} & \text{otherwise}.
\end{cases}
\end{align}
Here $\rvert \omega \wedge \nu \lvert := \min \, \{ \, n \in \mathbb{N} \colon \omega_{n} \neq \nu_n \}$, for all $\omega = (\omega_{1}\omega_{2}\dots) , \nu = ( \nu_{1} \nu_{2}\dots) \in \{1, 2\}^{\mathbb{N}}$ with $\omega \neq \nu$. If two sequences have the same first $N$ symbols, then their
interval addresses differ by at most $2^{-N}$.  Note that the topology induced by $d$ on $\{ 1, 2\}^{\mathbb{N}}$ coincides with the product topology on $\{ 1, 2\}^{\mathbb{N}}$.  For $n\in\mathbb{N}$ and $\omega\in\{1,2\}^{\mathbb{N}}$, we set $\omega|_{1}^{n}=\omega|_{n}=(\omega_{1}\cdots\omega_{n})$ and call $n$ the length of $\omega|_{n}$. We let $\sigma \colon \{ 1, 2 \}^{\mathbb{N}} \circlearrowleft$ denote the \textsl{left-shift map} which is defined by $\sigma(\omega_{1} \omega_{2} \dots) := (\omega_{2} \omega_{3}\dots)$. A \textsl{subshift} is any closed subset $\Omega \subseteq \{1,2\}^\mathbb{N}$ such that $\sigma(\Omega) \subseteq \Omega$.  Given a subshift $\Omega$ and $n \in \mathbb{N}$ we set
\begin{align}
\Omega\lvert_{n} := \left\{ (\omega_{1} \dots \omega_{n}) \in \{ 1, 2\}^{n} \colon \text{there exists} \, \omega \in \Omega \, \ \text{with} \ \, \omega|_{n} = (\omega_{1} \dots \omega_{n}) \right\}
\end{align}
and  write $\Omega^{*} := \bigcup_{n = 1}^{\infty} \Omega\lvert_{n}$ for the collection of all finite words.  For $\xi\in\Omega^{*}$, we denote $|\xi|$ as the length of $\xi$. And we denote by $\#\Omega|_{n}$ the cardinality of $\Omega|_{n}$.

\par For $n,m\in\mathbb{N}$, let
$\nu=(\nu_1\dots\nu_n)$ and $\xi=(\xi_1\dots\xi_m)$ be finite words over
$\{1,2\}$.  Set
\[
\nu\xi:=(\nu_1\dots\nu_n\xi_1\dots\xi_m).
\]
 We  use the same notation when $\xi \in \{ 1, 2\}^{\mathbb{N}}$.  An infinite word $\omega = (\omega_{1} \omega_{2} \dots) \in \{1, 2\}^{\mathbb{N}}$ is called \textsl{periodic} with \textsl{period} $n \in \mathbb{N}$ if and only if, $(\omega_{1}\dots \omega_{n}) = (\omega_{(m - 1)n + 1} \dots \omega_{m n})$, for all $m \in \mathbb{N}$; in which case we write $\omega = (\omega_{1} \dots \omega_{n})^{\infty}$.  Similarly,  $\omega = (\omega_{1}\omega_{2} \dots) \in \{1, 2\}^{\mathbb{N}}$ is called \textsl{eventually periodic} with \textsl{period} $n \in \mathbb{N}$ if and only if there exists $k \in \mathbb{N}$ such that $(\omega_{k+1} \dots \omega_{k+n}) = (\omega_{k+(m - 1)n + 1} \dots, \omega_{k+ m n})$ for all $m \in \mathbb{N}$; in which case we write $\omega = \omega_{1}\dots \omega_{k} (\omega_{k+1} \dots \omega_{k+n})^{\infty}$.

Let $f$ be a decreasing Lorenz map with critical point $c$.
On a decreasing branch, the shift exchanges the two one-sided itineraries:
\begin{equation}\label{eq:one-sided-shift}
\tau_f\bigl(f(x)\pm\bigr)=\sigma\bigl(\tau_f(x\mp)\bigr),
\qquad x\in(0,1)\setminus\{c\}.
\end{equation}
 At $x=c$  the
branch-limit identities at the discontinuity are written separately as
\[
\tau_f\bigl(f(c-)+\bigr)=\sigma(k_-),
\qquad
\tau_f\bigl(f(c+)-\bigr)=\sigma(k_+),
\]
whenever these one-sided image itineraries are needed.
This is the standard identity in
Glendinning~\cite[Section~3, Lemma~2]{Glendinning2014}, written in our
notation.

\subsection{Topological entropy}
The standard definition for the topological entropy of continuous maps \cite{de1993} using $(n,\epsilon)$-separated sets can be used to define entropy for piecewise continuous maps. An alternative way of calculating topological entropy, which is particularly convenient for the symbolic approach, is via the lap numbers \cite{milnor1988iterated} or the cardinality of finite words, i.e.,
$$ h_{\mathrm{top}}(f)=h_{\mathrm{top}}(\sigma|_{\Omega(f)})=\displaystyle{\lim_{n \to \infty}  \frac{\log \left( \# \Omega(f)|_{n}  \right)}{n}}=\displaystyle{\inf_{n}  \frac{\log \left( \# \Omega(f)|_{n}  \right)}{n}}.
$$
The limit exists by subadditivity of $\{\log(\#\Omega(f)|_n)\}_{n\in\mathbb N}$.
For simplicity, we denote $h_{\mathrm{top}}(\sigma|_{\Omega(f)})$ as $h_{\mathrm{top}}(\Omega(f))$ when we do not want to mention the left shift map $\sigma$.
Similarly, $h_{\mathrm{top}}(S_f(a,b))$ denotes the entropy of the
survivor restriction with the prescribed endpoint convention; for an
empty hole it denotes the full-system entropy $h_{\mathrm{top}}(f)$.
\par Here, we introduce the concept of the kneading determinant, which originates from \cite{milnor1988iterated}. Further developments can also be found in \cite{Glendinning1996} and \cite{barnsley2014}. Let $f$ be a decreasing Lorenz map with kneading invariants $(k_-,k_+)$, where $k_{+}=(v_{1}v_{2}\cdots)$ and $k_{-}=(w_{1}w_{2}\cdots)$. Then the {\bf kneading determinant} is a formal power series defined as $K(z) = {K_ + }(z) - {K_ - }(z)$ , where ${K_ + }(z) = \sum\limits_{i = 1}^\infty  {v_{i}{z^{i-1}}}$, and ${K_ - }(z) = \sum\limits_{i = 1}^\infty  {w_{i}{z^{i-1}}}.$

\par Barnsley et al.\ \cite{barnsley2014} work directly on the symbolic
space and do not require it to be the address space of a particular map.
They also treat the alternating case, which is directly applicable to our
setting. Recall that
\[
\Omega(\ba,\bb) = \{ w \in \{1,2\}^{\mathbb N} : \sigma^n w \notin (\ba,\bb)
\ \text{for all}\ n \ge 0\}.
\]

\begin{lemma}[{\bf\textup{\cite[Lemma 3]{barnsley2014}}}]\label{entropy}
Let $\ba = a_1a_2\cdots$ and $\bb = b_1b_2\cdots$ be sequences in
$\{1,2\}^{\mathbb N}$ with $a_1 = 1$ and $b_1 = 2$. Suppose that
$\ba, \bb \in \Omega(\ba,\bb)$  and that
$h_{\mathrm{top}}(\Omega(\ba,\bb)) > 0$. For the alternating lexicographic
order on $\Omega(\ba,\bb)$, set
\[
K(z) := \sum_{n=1}^\infty (b_n - a_n) z^{n-1}.
\]
Then $h_{\mathrm{top}}(\Omega(\ba,\bb)) = \log t'$, where $-1/t'$ is the
largest negative root of $K(z)$.
\end{lemma}

\par Raith~\cite{raith1989,raith1992,raith1994} treats piecewise monotonic
interval maps and survivor sets obtained by removing finitely many open
intervals.  Corollary~1.1 of~\cite{raith1994} gives continuity of the
survivor topological entropy with respect to the hole endpoints and does not
require expansion.  For an EDLM, the two branch restrictions are continuous
and strictly monotone, so the map is piecewise monotonic in Raith's sense.
Raith defines entropy using a continuous completion obtained by splitting
partition endpoints and their preimages; see~\cite[Section~1,
equation~(1.4)]{raith1994}. We verify below that this entropy agrees with
the survivor entropy used here, including the critical-endpoint and
empty-hole cases.

\begin{proposition}[{\cite[Corollary 1.1]{raith1994}}]\label{continuity}
Let $f\in\mathrm{EDLM}$ have critical point $c$ and a hole $H=(a,b)$.
For fixed $a\leq c$, the entropy map
$\lambda(a,\cdot):b\mapsto h_{\mathrm{top}}(S_f(H))$ is continuous on
$[c,1)$.
\end{proposition}

Given two sets $X$ and $Y$, denote by
$X \triangle Y := (X \setminus Y)\cup
(Y \setminus X)$ their symmetric difference.

\begin{lemma}\label{lem:countable-diff-entropy}
Let $X$ and $Y$ be compact subshifts over a finite alphabet. If
$X\triangle Y$ is at most countable, then
\[
h_{\mathrm{top}}(X)=h_{\mathrm{top}}(Y).
\]
\end{lemma}

\begin{proof}
We prove that $h_{\mathrm{top}}(Y)\geq h_{\mathrm{top}}(X)$; the reverse
inequality follows by interchanging $X$ and $Y$. If
$h_{\mathrm{top}}(X)=0$, there is nothing to prove. Otherwise, by the
variational principle, for every $\varepsilon>0$ there is an ergodic
$\sigma$-invariant probability measure $\mu$ supported on $X$ with
$h_\mu(\sigma)>h_{\mathrm{top}}(X)-\varepsilon>0$ (after reducing
$\varepsilon$ if necessary). A positive-entropy ergodic measure is
nonatomic, and hence $\mu(X\triangle Y)=0$. Since $\mu(X)=1$, it follows
that $\mu(Y)=1$. As $Y$ is closed and shift-invariant, $\mu$ is also an
invariant measure for $\sigma|_Y$, and therefore
\[
h_{\mathrm{top}}(Y)\geq h_\mu(\sigma)>h_{\mathrm{top}}(X)-\varepsilon.
\]
Letting $\varepsilon\downarrow0$ gives the desired inequality.
\end{proof}

\begin{proof}[Proof of Proposition~\ref{continuity}]
Take $K=1$ and $Q=(a,b)$ in Raith's notation. Suppose first that $a<b$.
The selected completion is continuous on $I\setminus(a,b)$, since any
retained critical endpoint is approached only from its selected side.
Its finite-time survivor sets, and hence $S_f(a,b)$, are compact.
Labelling the retained left and right components by $1$ and $2$ gives a
continuous itinerary map. Density of $C(f)$ makes this map injective:
two distinct points with the same itinerary would leave the interval
between them free of precritical points. Thus it conjugates the completed
survivor system to a compact subshift $X_H$.

Let $R_Y$ be Raith's compact survivor system in the split space, and let
$Y_H$ be its image under the two-branch itinerary map. This map is
continuous because the two sides of $c$ are separated. The same
expansivity argument shows that a fibre can contain only the two split
copies of a single interval point. This finite-to-one factor therefore
preserves entropy, so Raith's survivor entropy equals
$h_{\mathrm{top}}(Y_H)$.
Outside the partition and hole endpoints and their preimages, the two
systems have the same itineraries and the same hole-avoidance condition.
The exceptional points form a countable set, since each branch is
strictly monotone. Consequently, $X_H\triangle Y_H$ is at most countable,
and Lemma~\ref{lem:countable-diff-entropy} gives
\[
h_{\mathrm{top}}(S_f(a,b))
=h_{\mathrm{top}}(X_H)=h_{\mathrm{top}}(Y_H).
\]
For $a=b=c$, Raith's completion has the full kneading space $\Omega(f)$;
the same finite-to-one coding and the symbolic entropy formula above give
the entropy $h_{\mathrm{top}}(f)=h_{\mathrm{top}}(S_f(c,c))$.

Thus the two entropy conventions agree for all the holes under
consideration. Raith's parameter space permits coincident hole endpoints,
so Corollary~1.1 of~\cite{raith1994} gives continuity as $b$ varies in
$[c,1)$, including right-continuity at $b=c$ and at the empty hole.
\end{proof}

\subsection{Hausdorff dimension}\label{subsecdimensiondef}
For a set $K\subset\mathbb R$ and $s\geq0$, its $s$-dimensional
Hausdorff measure is defined by
\[
\mathcal H^s(K)=\lim_{\delta\downarrow0}
\inf\left\{\sum_i(\operatorname{diam}U_i)^s:
K\subseteq\bigcup_iU_i,\ \operatorname{diam}U_i\leq\delta\right\}.
\]
The Hausdorff dimension of $K$ is
\[
\dim_{\mathrm H}K:=\inf\{s\geq0:\mathcal H^s(K)=0\}.
\]

Let $T$ be a piecewise $C^{1+\alpha}$ expanding interval map and let
$K$ be a compact invariant set avoiding its discontinuities.  For a
continuous potential $\varphi$, write
\[
P_T(\varphi,K):=\sup_{\mu\in\mathcal M_T(K)}
\left(h_\mu(T)+\int_K\varphi\,d\mu\right),
\]
where $\mathcal M_T(K)$ denotes the set of $T$-invariant Borel
probability measures supported on $K$.  For the geometric potential
$\varphi_s=-s\log|T'|$, the standard Bowen formula states that, under
the usual expansion and bounded-distortion hypotheses,
\[
\dim_{\mathrm H}K=s_*,
\qquad
P_T(-s_*\log|T'|,K)=0.
\]
See, for example, Raith~\cite{raith1989,raith1994}.  We use this formula
below only in the expanding negative-doubling setting.  For
$T_{-2}$, the slope is constant on both branches, so
\[
P_{T_{-2}}(-s\log2,S_{-2}(a,b))
=h_{\mathrm{top}}(S_{-2}(a,b))-s\log2,
\]
and hence
\[
\dim_{\mathrm H}S_{-2}(a,b)
=\frac{h_{\mathrm{top}}(S_{-2}(a,b))}{\log2}.
\]

The requirement that \(K\) avoid the discontinuities is a standard
sufficient hypothesis for the Bowen formula: it makes the restriction
of \(T\) to \(K\) continuous and removes any ambiguity in the
geometric potential. For \(T_{-2}\), if \(a<c<b\), then \(c\) belongs
to the hole and \(S_{-2}(a,b)\) automatically avoids the discontinuity.
If \(a=c<b\), we use the left completion at \(c\); since
\(S_{-2}(c,b)\cap(c,b)=\varnothing\), the restriction of the completed
map to the compact survivor set is continuous in the relative topology.
The case \(a<c=b\) is symmetric, using the right completion. Thus the
same constant-slope pressure calculation applies to these nonempty
endpoint cases. In the empty-hole case \(a=b=c\), one has
\(S_{-2}(c,c)=[0,1]\) and \(\dim_{\mathrm H}S_{-2}(c,c)=1\) directly.
The critical point and its preimages form a countable set, so retaining
or deleting the corresponding one-sided boundary points does not change
Hausdorff dimension.

For a general EDLM, equality of symbolic survivor spaces alone does not
imply equality of Hausdorff dimensions; an additional geometric
regularity or distortion assumption is required.

\section{Circle embedding and realization of EDLMs}\label{secsec3}
\subsection{Expansive endomorphisms of the circle}
Let \(S^1=\mathbb R/\mathbb Z\), and define
\[
m_{-2}:S^1\to S^1,\qquad
m_{-2}(\theta)=-2\theta\pmod 1.
\]
 The proof of the main results involves embedding our topologically
expansive decreasing Lorenz maps into $m_{-2}$. For any sequence
$\omega=(a_1a_2\cdots)\in\{1,2\}^{\mathbb N}$, define
 $$ [\omega]
:=
-\sum_{i=1}^{\infty}\frac{a_i}{(-2)^i}
 \pmod 1
 \in S^1.
 $$
These two definitions are compatible in the form
\begin{equation}\label{eq:negative-address-shift}
m_{-2}([\omega])=[\sigma(\omega)].
\end{equation}
\begin{proposition}\label{po1}
 Let \(g:S^1\to S^1\) be a topologically expansive
 continuous orientation-reversing covering of degree \(-2\).
 Then \(g\) has exactly three fixed points.
For each fixed point \(x_0\), there exists a unique
orientation-preserving homeomorphism
\(\phi:S^1\to S^1\) such that
$
\phi(x_0)=0$ and $
\phi\circ g=m_{-2}\circ\phi.$
\end{proposition}

\begin{proof}
Since $g$ is locally orientation reversing, each fixed point is isolated
and has index $+1$. The Lefschetz number $L(g)=1-\deg(g)=3$ therefore
gives exactly three fixed points.

Fix one, $x_0$, and let $y$ be its other preimage. The counterclockwise
arcs $I_1=[x_0,y]$ and $I_2=[y,x_0]$ define itineraries
$k(x)=\epsilon_0\epsilon_1\cdots$, with $\epsilon_i=j$ when
$g^i(x)\in I_j$, using one-sided itineraries at partition boundaries.
At $x_0$ these are $(12)^\infty,(21)^\infty$, with interval addresses
$0,1$; at $y$ they are $\ell_0:=1(12)^\infty,r_0:=2(21)^\infty$,
both with address $1/2$. Further preimages have the two itineraries
$w\ell_0,wr_0$, again with equal addresses. Thus
$\phi(x):=[k(x)]$ is well defined and continuous. The counterclockwise
labeling and alternating order make $\phi$ monotone of degree $1$,
and~\eqref{eq:negative-address-shift} gives
$\phi(x_0)=0$ and $\phi\circ g=m_{-2}\circ\phi$.

Each fibre $J_z:=\phi^{-1}(z)$ is a proper closed arc or a point,
and $g(J_z)\subset J_{m_{-2}(z)}$. If $g|_{J_z}$ were not injective,
a strictly decreasing lift would give $g(J_z)=S^1$, a contradiction.
Every point of $J_{m_{-2}(z)}$ has two preimages in
$J_z\cup J_{z+1/2}$, one in each fibre by injectivity. Hence
$g:J_z\to J_{m_{-2}(z)}$ is a homeomorphism.
If $J_z$ is nontrivial and $z$ is eventually periodic, an iterate
restricts to a positively expansive homeomorphism of a nontrivial
closed interval, which is impossible. Otherwise its forward fibres
are pairwise disjoint and their lengths tend to zero; two distinct
points in a sufficiently late fibre then remain within an
expansiveness constant forever. Thus every fibre is a singleton and
$\phi$ is an orientation-preserving homeomorphism.

Any other normalized orientation-preserving conjugacy $\psi$ must
send $y$ to $1/2$ and $I_1,I_2$ to $[0,1/2],[1/2,1]$, respectively.
It therefore preserves itineraries, so $\psi(x)=[k(x)]=\phi(x)$.
\end{proof}

\subsection{\texorpdfstring{Constructing an $f\in\mathrm{EDLM}$ with given sequences}{Constructing an EDLM with given sequences}}\label{constr}

\begin{lemma}\label{lem:upward-jump}
Every topologically expansive decreasing Lorenz map satisfies
$f(c-)<f(c+).$
Consequently, its kneading invariants satisfy
$
\sigma(k_-)\prec\sigma(k_+).
$
\end{lemma}

\begin{proof}
The critical-value inequality is part of the definition. Density of
the preimages of $c$ makes the one-sided itinerary map strictly
order-preserving: an interval between distinct points with equal
itineraries would contain no preimage of $c$. The branch shift
identities therefore give
\[
\sigma(k_-)=\tau_f\bigl(f(c-)+\bigr)
\prec\tau_f\bigl(f(c+)-\bigr)=\sigma(k_+).
\]
\end{proof}

\begin{definition}\label{def:decreasing-admissible}
Let $k_-\in1\{1,2\}^{\mathbb N}$ and
$k_+\in2\{1,2\}^{\mathbb N}$.  The pair $(k_-,k_+)$ is called
\emph{decreasing-admissible} if
$
\sigma(k_-)\prec\sigma(k_+),
$
and, for every $j\geq0$,
\begin{equation}\label{eq:adm-even-odd}
\left\{
\begin{aligned}
\sigma^{2j}(k_-)&\notin(k_-,k_+],\qquad
\sigma^{2j}(k_+)\notin[k_-,k_+),\\
\sigma^{2j+1}(k_-)&\notin[k_-,k_+),\qquad
\sigma^{2j+1}(k_+)\notin(k_-,k_+].
\end{aligned}
\right.
\end{equation}
\end{definition}

The following proposition is the decreasing, parity-sensitive analogue of the
necessity part of Hubbard--Sparrow's Theorem~1~\cite{Hubbard1990}.
Their increasing case is expressed by a single pair of inequalities,
whereas the reversal of the two branches in the present setting requires
separate even- and odd-iterate conditions.

\begin{proposition}\label{prop:necessity}
If $f\in\mathrm{EDLM}$, then its kneading pair $(k_-,k_+)$ is decreasing
admissible.
\end{proposition}

\begin{proof}
Lemma~\ref{lem:upward-jump} gives $\sigma(k_-)\prec\sigma(k_+)$.
Strict order preservation of itineraries implies that no realized
itinerary lies in $(k_-,k_+)$, so all shifts of $k_-,k_+$ avoid this gap.
The one-sided shift identity~\eqref{eq:one-sided-shift}
\cite[Section~3, Lemma~2]{Glendinning2014} determines the possible
endpoint equalities: each decreasing branch reverses side.
Even iterates preserve side and thus exclude the cross returns
$\sigma^n(k_-)=k_+$ and $\sigma^n(k_+)=k_-$; odd iterates exchange
side and exclude the self-returns $\sigma^n(k_-)=k_-$ and
$\sigma^n(k_+)=k_+$. Together with gap avoidance, these are precisely
\eqref{eq:adm-even-odd}.
\end{proof}

Let
$
k_-\in 1\{1,2\}^{\mathbb N}$ and $
k_+\in 2\{1,2\}^{\mathbb N}$
be two sequences such that
$(k_-,k_+)$ is decreasing-admissible.
We construct an EDLM whose kneading invariants are exactly $(k_-,k_+)$.
We first separate the full negative-doubling case from the case of a
nonempty arc to be collapsed. Put
\[
\ell_0:=1(12)^\infty,\qquad r_0:=2(21)^\infty,
\qquad A:=\pi_{-2}(k_-),\quad B:=\pi_{-2}(k_+).
\]

\begin{lemma} \label{lem:address-injectivity}
Let $(k_-,k_+)$ be decreasing-admissible. Then
$0<A\leq1/2\leq B<1$, and $A=B$ if and only if
$(k_-,k_+)=(\ell_0,r_0)$. If $A<B$, the restriction of $\pi_{-2}$ to
$\Omega(k_-,k_+)$ is injective and strictly order-preserving.
\end{lemma}
\begin{proof}
The address map is non-decreasing in the alternating order, as follows by
comparing the first differing digit with the largest possible tail.
It satisfies
\[
\pi_{-2}(1\omega)=\frac{1-\pi_{-2}(\omega)}2,
\qquad
\pi_{-2}(2\omega)=\frac{2-\pi_{-2}(\omega)}2.
\]
Its unique addresses of $0$ and $1$ are $(12)^\infty$ and
$(21)^\infty$, respectively. Thus $A=0$ would make $\sigma(k_-)$ the
largest sequence, and $B=1$ would make $\sigma(k_+)$ the smallest,
contradicting $\sigma(k_-)\prec\sigma(k_+)$. The first symbols give
$A\leq1/2\leq B$, with equality $A=B$ precisely for
$(k_-,k_+)=(\ell_0,r_0)$.

To check injectivity, suppose that distinct sequences $u,v$ have the same
address.  Write \(u=w\alpha\) and \(v=w\beta\), where \(w\) is
their maximal common prefix. The prefix \(w\) acts on interval addresses
by an injective affine inverse branch \(F_w\), thus
\(\pi_{-2}(u)=\pi_{-2}(v)\) implies
\(\pi_{-2}(\alpha)=\pi_{-2}(\beta)\).  The two remaining sequences have the form
$1u'$ and $2v'$. The displayed identities then force
$\pi_{-2}(u')=0$ and $\pi_{-2}(v')=1$. Hence
$
\{u,v\}=\{w\ell_0,wr_0\}.
$
Since $\ell_0$ is the largest sequence beginning with $1$ and $r_0$ is
the smallest sequence beginning with $2$, we have
$
k_-\preceq\ell_0\prec r_0\preceq k_+.
$
When $A<B$, at least one of $\ell_0,r_0$ lies in $(k_-,k_+)$.
The shift-invariance of $\Omega(k_-,k_+)$ therefore prevents both
$w\ell_0$ and $wr_0$ from belonging to it. This proves injectivity,
and strict order preservation follows from monotonicity.
\end{proof}

If $(k_-,k_+)=(\ell_0,r_0)$, take $f_{(k_-,k_+)}=T_{-2}$; its kneading
pair is exactly $(\ell_0,r_0)$ and its critical values are $0$ and $1$.
For the collapse construction that follows, assume $A<B$, and regard
$A,B\in(0,1)$ also as circle points. The following proposition and its
proof are illustrated in Figure~\ref{fig2}.

\begin{proposition}\label{prop:nested}
Let \(J_0=(A,B)\) be the open arc not containing \(0\), where
\(A=[k_-]\) and \(B=[k_+]\), and suppose that \(k_-\) and \(k_+\)
satisfy the decreasing-admissibility condition. Let \(J_a\) and \(J_b\) be
connected components of \(m_{-2}^{-a}(J_0)\) and
\(m_{-2}^{-b}(J_0)\), respectively. Then either one of \(J_a,J_b\)
is contained in the other, or their closures are disjoint.
\end{proposition}

\begin{proof}
The decreasing-admissibility conditions give
\[
m_{-2}^{n}(A),\,m_{-2}^{n}(B)\notin J_0
\qquad\text{for every }n\geq0.
\]

Suppose that \(J_a\) and \(J_b\) intersect and that neither is
contained in the other. Without loss of generality, assume that
\(a<b\). The map \(m_{-2}^{a}\) sends \(J_a\) homeomorphically
onto \(J_0\), while it sends \(J_b\) onto a connected component of
\(m_{-2}^{-(b-a)}(J_0)\). Since \(J_a\) and \(J_b\) overlap but
neither contains the other, the open arc \(m_{-2}^{a}(J_b)\)
meets \(J_0\) but is not contained in \(J_0\). It must therefore
contain either \(A\) or \(B\) in its interior. Since
\(m_{-2}^{b-a}\) maps \(m_{-2}^{a}(J_b)\) homeomorphically onto
\(J_0\), the remaining \(b-a\) iterates send this endpoint into
the interior of \(J_0\), contradicting the conditions above.
Hence, whenever \(J_a\) and \(J_b\) intersect, one is contained
in the other.

Now suppose that \(J_a\cap J_b=\varnothing\). Since all the
components under consideration avoid \(0\), we may cut the circle
at \(0\) and regard them as open intervals in an ordinary interval.
Consequently, if their closures meet, then they meet at exactly
one point:
\[
\overline{J_a}\cap\overline{J_b}=\{x\}.
\]
If \(a=b\), two distinct components of \(m_{-2}^{-a}(J_0)\)
cannot share an endpoint, since that endpoint would have to be
mapped simultaneously to \(A\) and \(B\). Hence, after interchanging
\(J_a\) and \(J_b\) if necessary, we may assume that \(a<b\).

Since \(x\) is an endpoint of both components, set
\[
E:=m_{-2}^{a}(x)\in\{A,B\},
\qquad
F:=m_{-2}^{b}(x)\in\{A,B\}.
\]
Then
$
F=m_{-2}^{b-a}(E).
$
 Every shift of $k_-$ and $k_+$ belongs to
$\Omega(k_-,k_+)$. Since $A,B\in(0,1)$, a circle return to either
endpoint is also an equality of interval addresses.
Lemma~\ref{lem:address-injectivity} converts it into the corresponding
equality of symbolic sequences.
The two intervals approach the common endpoint $x$ from opposite sides.
Consequently, the iterate $m_{-2}^{b-a}$ preserves this endpoint pairing when
$b-a$ is even and reverses it when $b-a$ is odd: in the first case the two
endpoint images are the opposite endpoints of $J_0$, so $E\ne F$, whereas in
the second case they are the same endpoint, so $E=F$.
If \(a\) and \(b\) have the same parity, then \(b-a\) is even and
\(E\neq F\). Hence \(m_{-2}^{b-a}\) maps one endpoint of \(J_0\)
to the other, which is forbidden by the even part of
the decreasing-admissibility condition. If \(a\) and \(b\) have opposite parity, then
\(b-a\) is odd and \(E=F\). Hence \(m_{-2}^{b-a}\) fixes an
endpoint of \(J_0\), which is forbidden by the odd part of
the decreasing-admissibility condition. Both cases are impossible, so the closures are
disjoint.
\end{proof}

In the remainder of this subsection, we assume that
\((k_-,k_+)\) is decreasing-admissible and \(A<B\);
the full pair \((\ell_0,r_0)\), for which \(A=B=1/2\), was handled
separately above.
Let \(\mathcal J\) denote the family of all connected components of
the preimages of \(J_0\):
\[
\mathcal J
:=
\left\{
J:\ J\text{ is a connected component of }m_{-2}^{-n}(J_0)
\text{ for some }n\geq0
\right\}.
\]
By Proposition~\ref{prop:nested}, the closures of any two
members of \(\mathcal J\) are either nested or disjoint. Let
\(\mathcal M\) be the family of maximal elements of
$
\left\{\overline J:J\in\mathcal J\right\}$
with respect to inclusion. The members of \(\mathcal M\) are pairwise
disjoint compact arcs, and every \(J\in\mathcal J\) is contained in a
unique member of \(\mathcal M\). Moreover, every component of
\(m_{-2}^{-n}(J_0)\) has length
$
2^{-n}\operatorname{length}(J_0),
$
so the diameters of the preimage components tend to zero as
\(n\to\infty\).

By Proposition~\ref{prop:nested}, the absence of endpoint touching,
and the fact that the diameters of the preimage components tend to zero,
the maximal collapsed arcs form an upper semicontinuous monotone
decomposition of \(S^1\). Collapsing every arc in \(\mathcal M\) to a
single point and leaving all remaining points unchanged therefore gives
an upper semicontinuous monotone quotient. Hence the resulting quotient
space, denoted by \(X\), is homeomorphic to \(S^1\).
 This is the  Moore decomposition argument for the
circle; see Hubbard--Sparrow~\cite[Section~3]{Hubbard1990}.  Let
\[
q:S^1\longrightarrow X
\]
be the quotient map. Since \(\overline{J_0}\) is a maximal member of the decomposition, set
\[
c:=q(A)=q(B).
\]
Then
\[
q(\overline{J_0})=\{c\},
\qquad
q^{-1}(c)=\overline{J_0}.
\]

Finally, cut the quotient circle \(X\) at \(q(0)\), replacing \(q(0)\)
by two one-sided copies denoted by \(0\) and \(1\). Since
\(0\notin\overline{J_0}\), we have \(q(0)\neq c\). The resulting space
is a compact interval, denoted by \(\bar I\), and \(c\) is an interior
point of \(\bar I\). Away from the two copies of \(q(0)\), we continue
to use \(q\) for the corresponding quotient coordinates in \(\bar I\).

No collapsed arc contains $0$: its image at the defining preimage depth
would otherwise put the fixed point $0$ in $\overline{J_0}$.
Thus $q^{-1}(q(0))=\{0\}$. Cutting the source circle at $0$ as well gives
a continuous non-decreasing interval quotient
\[
Q:[0,1]\longrightarrow\bar I,\qquad Q(0)=0,\quad Q(1)=1,
\]
whose non-singleton fibres are precisely the arcs of $\mathcal M$,
viewed as closed subintervals of $(0,1)$; in particular,
$Q^{-1}(c)=[A,B]$. Define the candidate interval map away from $c$ by
\[
f_{(k_-,k_+)}(Q(x))=
\begin{cases}
Q(1-2x),&0\leq x<A,\\
Q(2-2x),&B<x\leq1.
\end{cases}
\]
Independence of the representative is checked in
Proposition~\ref{pro3.3}. The two relevant interval addresses are
\[
t_-:=\pi_{-2}(\sigma(k_-))=1-2A,
\qquad
t_+:=\pi_{-2}(\sigma(k_+))=2-2B.
\]

\begin{lemma}\label{lem:automatic-jump}
For every decreasing-admissible pair with $A<B$, the interval quotient
above satisfies
$
Q(t_-)<Q(t_+).
$
\end{lemma}
\begin{proof}
Lemma~\ref{lem:address-injectivity} and
$\sigma(k_-)\prec\sigma(k_+)$ give $t_-<t_+$. Their circle orbits
avoid $J_0$, since all shifts of the boundary sequences avoid
$(k_-,k_+)$. If $Q(t_-)=Q(t_+)$, the two points belong to a common
non-singleton fibre $\overline J\in\mathcal M$, where $J$ is a
component of $m_{-2}^{-n}(J_0)$. Neither lies in $J$, so they are its
left and right endpoints.

For even $n$, their images under $m_{-2}^n$ are $A,B$, respectively.
Since $A,B\in(0,1)$, these are interval-address equalities;
Lemma~\ref{lem:address-injectivity} gives
$\sigma^{n+1}(k_-)=k_-$ and $\sigma^{n+1}(k_+)=k_+$,
contradicting the odd part of~\eqref{eq:adm-even-odd}.
For odd $n$, the endpoint images are reversed, giving
$\sigma^{n+1}(k_-)=k_+$ and $\sigma^{n+1}(k_+)=k_-$,
contrary to its even part. Thus monotonicity of $Q$ gives
$Q(t_-)<Q(t_+)$.
\end{proof}

For a nonempty $J_0$, the circle images $q(m_{-2}(A))$ and
$q(m_{-2}(B))$ are also distinct. Otherwise the strict interval
inequality would identify them only across the cut, forcing
$Q(t_-)=0$, $Q(t_+)=1$. The singleton endpoint fibres then give
$t_-=0$, $t_+=1$, hence $A=B=1/2$, a contradiction.
For the full pair $(\ell_0,r_0)$, however, the two circle images coincide
at $0$, while the interval critical values are $0$ and $1$.
Thus separation must be interpreted in interval coordinates when this
case is included; see Remark~\ref{re1}.

\begin{proposition}\label{pro3.3}
Let $(k_-,k_+)$ be decreasing-admissible with \(A<B\).
The map induced by \(m_{-2}\),
\[
f_{(k_-,k_+)}:\bar I\longrightarrow\bar I,
\]
has the following properties:
\begin{enumerate}
\item[(i)]
The map \(f_{(k_-,k_+)}\) is continuous and strictly decreasing on
 \([0,c)\) and \((c,1]\).

\item[(ii)]
$
f_{(k_-,k_+)}(0)=1,
\
f_{(k_-,k_+)}(1)=0.
$

\item[(iii)]
The two one-sided limits at \(c\) exist and satisfy
\[
f_{(k_-,k_+)}(c-)=Q(t_-),
\qquad
f_{(k_-,k_+)}(c+)=Q(t_+).
\]
In particular,
$
f_{(k_-,k_+)}(c-)
<
f_{(k_-,k_+)}(c+).$
Moreover, the lower and upper kneading sequences of \(c\) are
\(k_-\) and \(k_+\), respectively.
\end{enumerate}
\end{proposition}

\begin{proof}
We first verify that \(f_{(k_-,k_+)}\) is well defined away from \(c\).
Suppose that
\[
q(\theta)=q(\theta')
\]
for two points
\(\theta,\theta'\in S^1\setminus\overline{J_0}\). If
\(\theta\neq\theta'\), then they belong to the same maximal collapsed
arc \(\overline J\), where \(J\) is a connected component of
\(m_{-2}^{-n}(J_0)\) for some \(n\geq1\). The image
\(m_{-2}(\overline J)\) is contained in the closure of a connected
component of \(m_{-2}^{-(n-1)}(J_0)\), and hence is contained in a
single collapsed arc. Therefore
$
q\bigl(m_{-2}(\theta)\bigr)
=
q\bigl(m_{-2}(\theta')\bigr).
$
Thus the definition is independent of the chosen representative.
 After identifying the outer interval endpoints, this is the
circle identity
\[
f_{(k_-,k_+)}\circ q
=
q\circ m_{-2}
\qquad
\text{on }S^1\setminus\overline{J_0}.
\]
It follows from the quotient property that \(f_{(k_-,k_+)}\) is
continuous away from \(c\).

\noindent\textup{(i)}\quad On each of the two exterior arcs adjacent to \(A\) and \(B\), the map
\(m_{-2}\) reverses orientation, while \(q\) preserves the induced
order. Hence \(f_{(k_-,k_+)}\) is non-increasing on both
\([0,c)\) and \((c,1]\).

Suppose that \(x<y\) lie on the same side of \(c\) and have the same
image. Choose lifts \(\theta_x,\theta_y\) such that the arc \(K\) between
them does not meet \(\overline{J_0}\). If the common image belongs to
a singleton fibre of \(q\), then
$
m_{-2}(\theta_x)=m_{-2}(\theta_y).
$
Since \(K\) lies in one inverse branch of \(m_{-2}\), this gives
\(\theta_x=\theta_y\), hence \(x=y\).

Otherwise, \(q\circ m_{-2}\) is monotone on \(K\) and has equal endpoint
values. Hence \(q(m_{-2}(K))\) is a singleton, so
\(m_{-2}(K)\subset M\) for some non-singleton fibre
\(M\in\mathcal M\). Therefore
\(K\subset m_{-2}^{-1}(M)\); since \(K\) lies in one inverse branch,
it is contained in a connected component of \(m_{-2}^{-1}(M)\), and this
component is contained in a maximal collapsed interval. Thus \(q(K)\) is
a singleton, contradicting \(x<y\). Therefore
\(f_{(k_-,k_+)}\) is strictly decreasing on both branches.

\par\medskip
\noindent\textup{(ii)}\quad The points \(0\) and \(1\) are the two one-sided copies of
\(q(0)\). Since \(m_{-2}\) reverses orientation, the endpoint values are
\[
f_{(k_-,k_+)}(0)=1,\qquad f_{(k_-,k_+)}(1)=0.
\]

\par\medskip
\noindent\textup{(iii)}\quad  By the defining branch formulas and continuity of \(Q\),
\[
f_{(k_-,k_+)}(c-)=Q(1-2A)=Q(t_-),
\qquad
f_{(k_-,k_+)}(c+)=Q(2-2B)=Q(t_+).
\]
The strict inequality \(Q(t_-)<Q(t_+)\) is exactly
Lemma~\ref{lem:automatic-jump}.

Finally, decreasing-admissibility implies that the forward orbits of
\(A\) and \(B\) do not enter \(J_0\). Under the negative base
\(-2\) address coding, the exterior one-sided itinerary of \(A\) is
\(k_-\), while that of \(B\) is \(k_+\). The quotient construction
preserves these one-sided itineraries. Hence the lower and upper
 kneading sequences of \(c\) are \(k_-\) and \(k_+\), respectively.
\end{proof}

It remains to prove topological expansivity of $f_{(k_-,k_+)}$.

\begin{figure}[t]
\begin{center}
\includegraphics[width=0.5\textwidth]{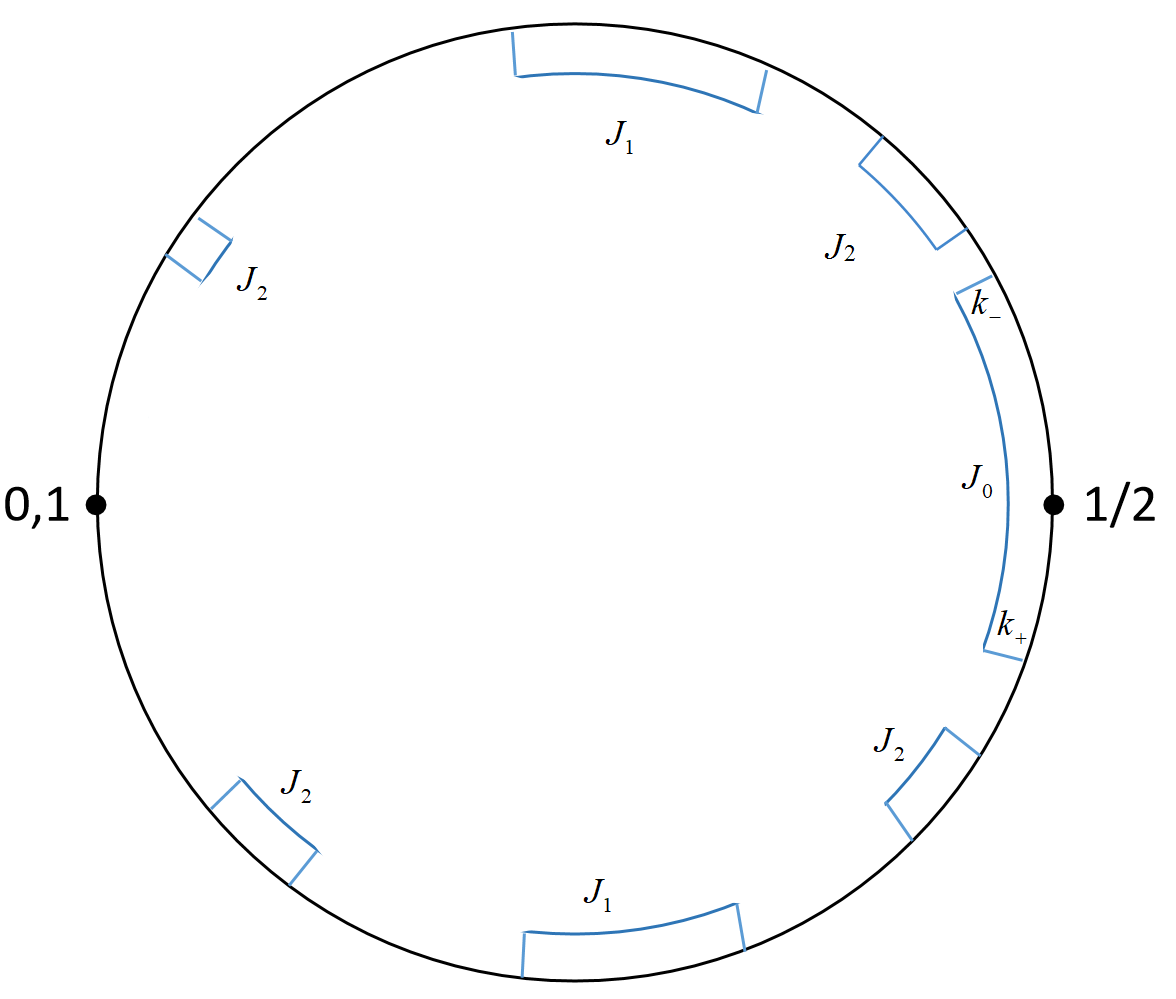}
\caption{The closures of the preimage components of $J_0=(A,B)$
under $m_{-2}$ are either nested or disjoint.
Collapsing the maximal such closed arcs gives the quotient circle $X$,
with $q(\overline{J_0})=\{c\}$.
Cutting $X$ at $q(0)$ gives the interval supporting the decreasing
Lorenz map with kneading invariants $(k_-,k_+)$.}
\label{fig2}
\end{center}
\end{figure}

\begin{proposition}\label{prop:expansive}
The map $f_{(k_-,k_+)}:\bar I\to\bar I$  constructed from any
decreasing-admissible pair, including the full pair $(\ell_0,r_0)$,
is topologically expansive.
\end{proposition}
\begin{proof}
The full pair gives $T_{-2}$, whose preimages of $1/2$ are dense.
Assume $A<B$, and choose a nonempty open subinterval $U_0$ of any
nonempty open set $U\subset\bar I$, avoiding the cut point.
The saturated open set $q^{-1}(U_0)$ meets some preimage component
of $J_0$, because $m_{-2}$ is expanding. It therefore contains the
maximal collapsed arc $M$ containing that component, so $z=q(M)\in U_0$.
Let $r$ be its first-arrival depth: $M$ is the closure of a component
of $m_{-2}^{-r}(J_0)$ and $m_{-2}^j(M)\cap J_0=\varnothing$ for
$0\leq j<r$. The branch semiconjugacy gives
\[
f_{(k_-,k_+)}^r(z)=q(m_{-2}^r(M))=q(\overline{J_0})=c.
\]
Thus the preimages of $c$ meet every $U$, proving expansivity in
the sense of Definition~\ref{defofdecreasing}.
\end{proof}

\begin{remark}\label{re1}
 For the full pair $k_- = \ell_0=1(12)^\infty$ and
$k_+=r_0=2(21)^\infty$, both interval addresses equal $1/2$.
Thus $J_0$ is empty and no arcs are collapsed. Cutting the circle at $0$
turns $m_{-2}$ into $T_{-2}$, whose kneading space is the full shift.
Here $t_-=0$ and $t_+=1$, so the interval map has a strict upward jump
even though $m_{-2}(A)=m_{-2}(B)=0$ on the circle. This is an empty-hole
case, not a failure of separation of the interval critical values.
\end{remark}

\subsection{\texorpdfstring{Embedding a map $f\in\mathrm{EDLM}$ on the circle}{Embedding an EDLM on the circle}}\label{sec2.3}
We embed an EDLM into an expansive degree-$-2$ circle map by reversing
the collapse in Section~\ref{constr}.

\par
For \(k\geq0\), let
\[
C_k:=\{x\in[0,1]: f^j(x)\ne c\ (0\leq j<k),\ f^k(x)=c\},
\]
so \(C_0=\{c\}\).
We adapt the interval-insertion construction of
Hubbard--Sparrow~\cite[Section~4]{Hubbard1990} to the decreasing case.
If $(k_-,k_+)\ne(\ell_0,r_0)$, choose $1<\gamma<2$ such that
\[
L:=\sum_{k\geq0}\#C_k\gamma^{-k}<\infty;
\]
the existence of this choice is proved below. Replace each $x\in C_k$
by a closed interval $J_x$ of length $\gamma^{-k}$, using only the
cumulative inserted lengths as the new coordinate. Denote the ordered
completion by $\widetilde I=[0,L]$, its monotone collapse onto $[0,1]$
by $\pi$, and the interval replacing $c$ by $J_0=[e_-,e_+]$.
Thus $e_+-e_-=1$.
For the full pair $(\ell_0,r_0)$, instead put
$\widetilde I=[0,1]$, $\pi=\mathrm{id}$, and $J_0=\{c\}$,
with $e_-=e_+=c$; no intervals are inserted.

\begin{lemma} \label{lem:blow-up-completion}
Let $f\in\mathrm{EDLM}$. Identifying the endpoints of $\widetilde I$
gives a circle $\widetilde S^1$ carrying a continuous, positively
expansive, orientation-reversing covering $\widetilde f$ of degree $-2$.
In cut interval coordinates, its exterior branches $\widetilde f_-$
and $\widetilde f_+$ satisfy
\[
\pi\circ\widetilde f_\pm=f\circ\pi
\quad\text{on their respective domains outside }J_0,
\]
with endpoint values
\[
\pi\bigl(\widetilde f_-(e_-)\bigr)=f(c-),\qquad
\pi\bigl(\widetilde f_+(e_+)\bigr)=f(c+).
\]
In the nonfull case, for $x\in C_k$, $k\geq1$, the map sends $J_x$
affinely and orientation-reversingly onto $J_{f(x)}$.
\end{lemma}

\begin{proof}
First suppose that $(k_-,k_+)\ne(\ell_0,r_0)$.
The proof of Proposition~\ref{prop:necessity} shows that every realized
itinerary and each of its shifts avoids $(k_-,k_+)$. In the nonfull
case this gap contains a cylinder determined by a finite word $v$,
say of length $p$. Each point of $C_k$ is determined by its length-$k$
branch itinerary, which avoids $v$. Writing $k=qp+r$, $0\leq r<p$,
we therefore have
\[
\#C_k\leq(2^p-1)^q2^r.
\]
Thus any $\max\{1,(2^p-1)^{1/p}\}<\gamma<2$ makes $L$ finite.
Since $C(f)$ is dense, the cumulative inserted lengths give the
ordered interval and continuous collapse described above; the
complement of the interiors of the inserted intervals has length zero.

For $x\in C_k$, $k\geq1$, map $J_x$ decreasingly and affinely onto
$J_{f(x)}$. These maps extend by order to the two exterior branches.
For an inserted target fibre $J_y$ lying in the image of one exterior
branch, the inverse branch gives a unique $x$ with $f(x)=y$; if
$y\in C_j$, then $x\in C_{j+1}$. Thus the inserted fibres meeting a
domain interval $[u,v]$ in that branch correspond bijectively, in
reverse order, to the inserted fibres meeting
$[\widetilde f(v),\widetilde f(u)]$. At $e_-$ use the right endpoint
of the fibre over $f(c-)$, and at $e_+$ the left endpoint of the
fibre over $f(c+)$; this includes partial endpoint fibres and gives
the prescribed one-sided values continuously.

The residual sets in the domain and image branches have length zero, so
the image interval has no additional positive-length part. Since each
fibre $J_x$ is mapped affinely onto $J_{f(x)}$, including its partial
endpoint pieces, for every $u<v$ on one exterior branch we have
\[
\begin{aligned}
\left|\widetilde f(v)-\widetilde f(u)\right|
&=\sum_{x\in C(f)\setminus\{c\}}
  \left|\widetilde f\bigl([u,v]\cap J_x\bigr)\right|\\
&=\gamma\sum_{x\in C(f)\setminus\{c\}}
  \left|[u,v]\cap J_x\right|
=\gamma(v-u).
\end{aligned}
\]
Consequently, both exterior branches are affine with slope
$-\gamma$. Since $f(0)=1$ and $f(1)=0$, they are
\[
\widetilde f_-(t)=L-\gamma t,\qquad
\widetilde f_+(t)=\gamma(L-t).
\]

After identifying $0$ and $L$, take the lift $F$ equal to
$\widetilde f_-$ on $[0,e_-]$ and to $\widetilde f_+-L$ on $[e_+,L]$,
and interpolate affinely across $J_0$. Its absolute slope on $J_0$ is
\[
s=2L-\gamma(L-1)=\gamma+(2-\gamma)L>\gamma>1.
\]
Thus $F(0)=L$, $F(L)=-L$, and extension by
$F(t+L)=F(t)-2L$ gives a continuous strictly decreasing lift of degree
$-2$. All its pieces expand lengths by a factor at least $\gamma>1$,
so the induced circle covering is positively expansive.

For the full pair, strict order preservation of one-sided itineraries
gives $f(c-)=0$ and $f(c+)=1$. After identifying $0$ and $1$, define
$\widetilde f$ by $f(x)\pmod 1$ for $x\ne c$ and by the common circle
limit $0\pmod 1$ at $c$, independently of the assigned interval value
$f(c)$. This defines a continuous degree $-2$ circle covering with two
full branches and the required exterior branch correspondence.
The full inverse-cylinder partitions have mesh tending to zero,
because their boundaries include the dense precritical set of $f$.
Matching these ordered partitions with those of $m_{-2}$ gives an
orientation-preserving conjugacy, hence positive expansivity.
\end{proof}

For the remainder of this subsection, let $\widetilde f$ on
$\widetilde S^1$ be the circle map supplied by
Lemma~\ref{lem:blow-up-completion}. We retain the notation
$\pi:\widetilde I\to[0,1]$ for the collapse map and $J_0$ for the
fibre over $c$, which is a singleton in the full case.

\begin{proposition}\label{por2}
For each chosen fixed point $x_0$ of $\widetilde f$, there exists a unique
orientation-preserving homeomorphism $\phi_{x_0}:\widetilde S^1\to S^1$
such that
\[
\phi_{x_0}(x_0)=0,
\qquad
\phi_{x_0}\circ\widetilde f=m_{-2}\circ\phi_{x_0}.
\]
\end{proposition}
\begin{proof}
Apply Proposition~\ref{po1} to the circle covering supplied by
Lemma~\ref{lem:blow-up-completion}, normalized at $x_0$.
\end{proof}
For the interval realization, normalize the fixed point $b$ obtained
by identifying the outer endpoints, and put $\phi=\phi_b$.
In the nonfull case, the branch labels give
$\phi(e_-)=[k_-]=A$ and $\phi(e_+)=[k_+]=B$.
For $x\in C_k$, $\operatorname{int}J_x$ first enters
$\operatorname{int}J_0$ at time $k$, while the complement of the
fibre interiors is forward invariant.
Hence $\phi$ maps the inserted fibres onto the closures of the
components of
$\bigcup_{n\geq0}m_{-2}^{-n}((A,B))$.

\begin{remark}
The nonfull construction also applies when exactly one of
$k_-=\ell_0$ and $k_+=r_0$ holds. Only the full pair $(\ell_0,r_0)$
requires the construction without inserted intervals. In every case,
the circle is formed by identifying the two endpoints of the interval.
\end{remark}

 \begin{corollary}\label{cor:no-odd-period}
Let $(k_-,k_+)$ be the kneading pair of a map $f\in\mathrm{EDLM}$, and set
$
A=[k_-],
$ $
B=[k_+].$
Then neither $A$ nor $B$ is periodic under an odd iterate of
$m_{-2}$.
\end{corollary}

\begin{proof}
For the full pair \((\ell_0,r_0)\), we have \(A=B=1/2\) and
\(m_{-2}^n(1/2)=0\) for every \(n\geq1\), so the assertion is immediate.
Otherwise, Proposition~\ref{prop:necessity} and
Lemma~\ref{lem:address-injectivity} give $0<A<B<1$ and injectivity
of $\pi_{-2}$ on $\Omega(k_-,k_+)$. If $m_{-2}^p(A)=A$ for odd $p$,
the address shift identity~\eqref{eq:negative-address-shift} gives
\[
\pi_{-2}(\sigma^p k_-)=\pi_{-2}(k_-).
\]
Both sequences belong to $\Omega(k_-,k_+)$, so $\sigma^p k_-=k_-$,
contrary to the odd part of equation \eqref{eq:adm-even-odd}. The argument for
$B$ is identical.
\end{proof}
\subsection{Proof of Theorems \ref{thm1} and \ref{thm2}}

{\bf \noindent Proof of Theorem \ref{thm1}}
\par The first part of Theorem~\ref{thm1} (necessity) is exactly
Proposition~\ref{prop:necessity}.
 The corresponding normalized circle realization and its
fibre correspondence are established in Section~\ref{sec2.3}.
\par
For the converse, let \((k_-,k_+)\) be decreasing-admissible. If it is
the full pair \((\ell_0,r_0)\), take \(f=T_{-2}\). Otherwise
Lemma~\ref{lem:address-injectivity} gives \(A<B\), so the collapse
construction applies. Lemma~\ref{lem:automatic-jump} gives the strict
ordering of the two critical values, Proposition~\ref{pro3.3} gives the
decreasing branches and the prescribed kneading pair, and
Proposition~\ref{prop:expansive} gives dense preimages of \(c\). Hence
\(f\in\mathrm{EDLM}\).

For uniqueness, let $f,g\in\mathrm{EDLM}$ have the same kneading pair.
In the full case, Lemma~\ref{lem:blow-up-completion} gives increasing
interval homeomorphisms identifying their two branches and one-sided
limits with those of $T_{-2}$. In the nonfull case, use the circle
conjugacies normalized at the respective points obtained by identifying
the outer endpoints. Then
\[
H:=\phi_g^{-1}\circ\phi_f,\qquad
H\circ\widetilde f=\widetilde g\circ H.
\]
The common kneading pair determines the same $A=[k_-]$ and $B=[k_+]$.
By the fibre correspondence following Proposition~\ref{por2}, both
families of inserted fibres map to the closures of the maximal components
of $\bigcup_{n\geq0}m_{-2}^{-n}((A,B))$. Hence $H$ matches these fibres,
including the central ones. Cutting at the normalized points and
collapsing the fibres gives an increasing interval homeomorphism $h$.
The branch semiconjugacies in Lemma~\ref{lem:blow-up-completion} yield
$h\circ f=g\circ h$ off the critical point, with the corresponding
one-sided identities there.
\hfill\(\square\)

\vspace{0.5cm}
{\bf \noindent Proof of Theorem \ref{thm2}}
\par
By Theorem~\ref{thm1}, an increasing interval homeomorphism identifies
the two branches and one-sided limits of $f$ with those of the canonical
realization $f_{(k_-,k_+)}$, with $T_{-2}$ used for the full pair.
This identification preserves upper and lower itineraries, so it suffices
to prove the assertions for this realization.

For the full pair \((\ell_0,r_0)\), no intervals are inserted, and the
itineraries of $T_{-2}$ give the full shift; the only pairs of sequences
representing the same point are
\(w\ell_0\) and \(wr_0\), at precritical points. Assume henceforth that
\(A<B\). Let $J_0=(A,B)$. Set
\[
K:=S^1\setminus\bigcup_{n\ge0}m_{-2}^{-n}(J_0).
\]
By Proposition~\ref{prop:nested}, the closures of the removed components
are nested or disjoint and do not touch end to end. Their lengths are
\(2^{-n}|J_0|\), and their union is dense because \(m_{-2}\) is expanding.
Thus \(K\) is closed and has empty interior. It has no isolated points,
since an isolated point would force two maximal removed intervals to have
closures meeting there, contrary to Proposition~\ref{prop:nested}. Hence
\(K\) is a Cantor set whose points {represent points of} the quotient interval.

For any sequence \(\omega\), its circle address satisfies
\(m_{-2}^n([\omega])=[\sigma^n(\omega)]\) by
\eqref{eq:negative-address-shift}. Since \([k_-]=A\) and \([k_+]=B\),
monotonicity in alternating order (Lemma~\ref{lem:address-injectivity})
gives
\[
[\omega]\in J_0
\quad\Longleftrightarrow\quad
\omega\in(k_-,k_+),
\qquad [\omega]\notin\{A,B\}.
\]
At \(A,B\) and their preimages, we use the exterior one-sided sequences
specified in Proposition~\ref{pro3.3}; two sequences of the same geometric
point need not both satisfy gap avoidance.
For every \(\omega\in\Omega(k_-,k_+)\), monotonicity and the address
shift relation give \([\omega]\in K\). After cutting at \(0\), the
quotient preserves this itinerary with the preceding one-sided choices
(Proposition~\ref{pro3.3}), so \(\omega\) is realized as an upper or lower
itinerary. Conversely, the alternating order and the one-sided shift
identity~\eqref{eq:one-sided-shift} imply that every realized itinerary
avoids \((k_-,k_+)\) under all shifts. Hence
$
\Omega(f)=\Omega(k_-,k_+).
$

Finally, after cutting at \(0\), two distinct points of \(K\) are identified
precisely when they are the endpoints of a maximal removed interval.
Their itineraries have a common finite prefix \(w\), followed, in either
order, by \(k_-\) and \(k_+\). Thus the only nontrivial identification
in passing from this cut set to the quotient interval is
$
wk_-\sim wk_+,
$
which represents the two sides of the same collapsed fibre.
\hfill\(\square\)

\section{Survivor sets of negative doubling maps}\label{secsec4}
We study $T_{-2}$ with a hole $(a,b)$, where
$0<a\leq1/2\leq b<1$, and write $c=1/2$ throughout this section.
For $a=c<b$, the survivor set $S_{-2}(c,b)$ uses the left completion,
with $T_{-2}(c-)=0$ and $\tau_{-2}(c-)=\ell_0=1(12)^\infty$.
For $a<c=b$, it uses the right completion, with
$\tau_{-2}(c+)=r_0=2(21)^\infty$. For $a<c<b$ the convention is
irrelevant, while $a=b=c$ retains the full kneading space
(Remark~\ref{re1}). Let $\ba$ and $\bb$ be the lower and upper boundary
sequences. The bifurcation parameter range is $[1/2,1)$.
Associate with $S_{-2}(a,b)$ the subshift determined by its boundary sequences:
\[
\Omega(\ba,\bb) = \{ w \in \{1,2\}^{\mathbb N} \,:\, \sigma^n w \notin (\ba,\bb) \ \mbox{for all}\ n \ge 0\}.
\]

\begin{remark}\label{rem:left-completion}
In the endpoint case \(a=c<b\), the lower boundary sequence is
\(\ba=\ell_0\), so the subshift determined by the boundary sequences
is \(\Omega(\ell_0,\bb)\). The successive shifts of \(\ell_0\) describe
the one-sided orbit \(c-\mapsto0\mapsto1\mapsto0\mapsto\cdots\);
the initial sequence is the left endpoint of the open forbidden interval,
while the subsequent orbit points lie outside \((c,b)\). If a point
first reaches \(c\) after finitely many iterates, its sequence is obtained
by appending the same tail \(\ell_0\). Thus the alternative tail
\(r_0\), corresponding to \(c+\), is not selected at the left
boundary.
The compact itinerary subshift of the completed interval survivor
system can differ from \(\Omega(\ell_0,\bb)\) at precritical
itineraries. Proposition~\ref{proptwoholes} shows that these exceptions
are at most countable and that the entropies agree.
\end{remark}

For general decreasing Lorenz maps, without assuming expansivity, we call
$(\ba,\bb)$ \emph{weak-admissible} if
\[
\ba,\bb\in\Omega(\ba,\bb).
\]
The lower boundary sequence $\ba$ is called \emph{self-admissible} when
$\ba\in\Omega(\ba,2(21)^\infty)$, and the upper boundary sequence $\bb$ is
\emph{self-admissible} when $\bb\in\Omega(1(12)^\infty,\bb)$. These conditions allow endpoint orbits.
If $\Omega(\ba,\bb)$ consists exactly of the two extremal sequences
$(12)^\infty$ and $(21)^\infty$, then the geometric survivor set is
$\{0,1\}$: only the endpoint orbit $0\mapsto1\mapsto0$ survives.
\subsection{Realization of decreasing Lorenz maps}\label{DLM realization}
The following lemma realizes symbolic survivor spaces without assuming
expansivity.
\begin{lemma}\label{lem:dlm-realization}
Let
\[
k_-\in1\{1,2\}^{\mathbb N},\qquad k_+\in2\{1,2\}^{\mathbb N}.
\]
Suppose that $(k_-,k_+)$ is weak-admissible and
$\sigma(k_-)\prec\sigma(k_+).$
Then there exists a decreasing Lorenz map
\(g\in\mathrm{DLM}\) whose kneading invariants are $(k_-,k_+)$ and whose
kneading space is
$
\Omega(g)=\Omega(k_-,k_+).
$
\end{lemma}

\begin{proof}
Put $\Omega=\Omega(k_-,k_+)$. This is a compact shift-invariant space
whose topology agrees with the alternating order, with least and greatest
elements $e_0=(12)^\infty$ and $e_1=(21)^\infty$. Write
$\Omega_i=\{\omega\in\Omega:\omega_1=i\}$. Weak-admissibility gives
$\max\Omega_1=k_-$ and $\min\Omega_2=k_+$, so $k_-,k_+$ are adjacent.

\noindent\textup{(i)}
Choose a countable dense set $D_0\subseteq\Omega$ containing
$e_0,e_1,k_-,k_+$ and both members of every adjacent pair. Such pairs
are countable: each is determined by its maximal common finite prefix
as the facing endpoints of the corresponding symbol cylinders. Set
\[
D=\bigcup_{m,n\geq0}\sigma^{-m}\bigl(\sigma^n(D_0)\bigr),
\]
with all preimages taken in $\Omega$. Since the shift has at most two
preimages, $D$ is countable; it is closed under shifts and allowed
preimages, so $\omega\in D\Longleftrightarrow\sigma(\omega)\in D$.
Choose a probability measure $\mu$ concentrated on $D$, with positive
mass at every element of $D$, and set
\[
L(\omega)=\mu\{\nu\in\Omega:\nu\prec\omega\},\qquad
R(\omega)=\mu\{\nu\in\Omega:\nu\preceq\omega\},\qquad
I_\omega=[L(\omega),R(\omega)].
\]
For $x\in(0,1)$, compactness and right-continuity of $R$ give a least
$\omega_x$ with $R(\omega_x)\geq x$. If $\omega_x\ne e_0$, minimality gives
$
L(\omega_x)=\sup_{\nu\prec\omega_x}R(\nu)\leq x\leq R(\omega_x);
$
if $\omega_x=e_0$, the left inequality follows from $L(e_0)=0$.
Thus $x\in I_{\omega_x}$, and $0\in I_{e_0}$, $1\in I_{e_1}$ prove coverage
of $[0,1]$. For $\omega\prec\nu$, density of $D$ gives
$R(\omega)\leq L(\nu)$, with equality exactly when the labels are adjacent.
Hence the intervals have disjoint interiors and meet precisely at
adjacent endpoints; they are nondegenerate exactly for $\omega\in D$.
In particular,
\[
c:=R(k_-)=L(k_+)\in(0,1),\qquad
\bigcup_{\omega\in\Omega_1}I_\omega=[0,c],\quad
\bigcup_{\omega\in\Omega_2}I_\omega=[c,1],
\]
where $0<c<1$ follows from the positive masses at $k_-$ and $k_+$.

\noindent\textup{(ii)}
For $\eta\in\Omega$, gap avoidance gives
$1\eta\in\Omega_1\Longleftrightarrow1\eta\preceq k_-$ and
$2\eta\in\Omega_2\Longleftrightarrow2\eta\succeq k_+$.
Deleting a common first symbol reverses the alternating order, so
\[
\sigma(\Omega_1)=\{\eta\in\Omega:\eta\succeq\sigma(k_-)\},\qquad
\sigma(\Omega_2)=\{\eta\in\Omega:\eta\preceq\sigma(k_+)\}.
\]
For $\omega\in D\cap\Omega_i$, map $I_\omega$ onto
$I_{\sigma(\omega)}$ by the decreasing affine homeomorphism; for
$\omega\notin D$, map its singleton fibre to the corresponding singleton.
Adjacent $\omega\prec\nu$ in one branch have adjacent images in reverse
order, so both prescriptions send their shared endpoint to
$L(\sigma(\omega))=R(\sigma(\nu))$. By coverage and reversal of order,
these pieces form strictly decreasing bijections
\[
G_1:[0,c]\longrightarrow[L(\sigma(k_-)),1],\qquad
G_2:[c,1]\longrightarrow[0,R(\sigma(k_+))],
\]
which are homeomorphisms because their domains and ranges are compact
intervals. {Define \(g:I\to I\) by \(g=G_1\) on \([0,c)\),
\(g=G_2\) on \((c,1]\), and choose any \(g(c)\in I\).
The kneading sequences are determined by the two branch limits and
are independent of this last choice.}
Since $\sigma(e_0)=e_1$ and $\sigma(e_1)=e_0$, we have $g(0)=1$, $g(1)=0$.
The strict order $\sigma(k_-)\prec\sigma(k_+)$ and the positive masses
at these two sequences give
\[
g(c-)=L(\sigma(k_-))<R(\sigma(k_+))=g(c+).
\]
Thus $g\in\mathrm{DLM}$.

\noindent\textup{(iii)}
For $\omega\in D$, the affine definition maps
$\operatorname{int}I_\omega$ onto $\operatorname{int}I_{\sigma(\omega)}$;
iteration shows that every point of this interior has itinerary $\omega$.
For $\omega\notin D$, all shifts remain outside $D$, so their fibres are
singletons and cannot be shared endpoints, whose labels belong to $D$.
The same argument therefore realizes $\omega$ at its singleton fibre.
At a shared endpoint, approaching through the adjacent fibre interiors
gives their labels as the two one-sided itineraries. In particular, the
kneading pair at $c$ is $(k_-,k_+)$, and every sequence in $\Omega$ is realized.

Conversely, the itinerary map is nondecreasing in the alternating order:
at the first differing symbol, spatial and symbolic order have undergone
the same number of reversals. Thus every realized one-sided itinerary
beginning with $1$ is at most $k_-$, and one beginning with $2$ is at least
$k_+$. By the one-sided shift identity~\eqref{eq:one-sided-shift} and the
branch-limit identities at $c$, these inequalities hold for every shift.
Hence every realized itinerary belongs to $\Omega$, proving
$\Omega(g)=\Omega(k_-,k_+)$.
\end{proof}

\subsection{Symbolic adjustment lemmas}
 We first record the right-boundary adjustment for an arbitrary fixed left
  endpoint. The argument is symbolic.
\begin{lemma}\label{lelele11}
Let $T_{-2}$ have a hole $(a,b)$ with $0\leq a\leq 1/2$, and keep
$a$ fixed while $b$ varies in $(1/2,1)$. Let $\ba$ and $\bb=(v_1v_2\cdots)$ be the
lower and upper boundary sequences, respectively.  If $\bb$ is not
self-admissible, let $n$ be the minimal integer such that
$\sigma^n(\bb)\prec\bb$ with $v_{n+1}=2$, and define
$\bb'=(v_1\cdots v_n)^\infty$.  Then
\begin{enumerate}
    \item $\bb'$ is self-admissible, i.e., $\bb' \in \Omega( 1(12)^\infty, \bb')$;
    \item if $n$ is even, $\bb'\succ\bb$ and $\Omega(\ba, \bb) = \Omega(\ba, \bb')$;
    \item if $n$ is odd, $\bb'\prec\bb$ and $\Omega(\ba, \bb') \setminus \Omega(\ba, \bb)$ is countable.
\end{enumerate}
\end{lemma}
\begin{proof}
The proof is independent of the numerical value of $a$.  Indeed, every
lower boundary sequence begins with $1$, whereas the first violating shift
$\sigma^n(\bb)$ begins with $2$.  Thus this violating shift lies above the
fixed lower boundary and enters the forbidden interval $(\ba,\bb)$ for every
fixed $a\leq 1/2$.  The remaining comparisons use only the common prefix
$v_1\cdots v_n$, its parity, and the fact that $\ba$ is kept fixed.
\par
\rm (1) We proceed by contradiction. Suppose $\bb'$ is not self-admissible. Then
there exists a minimal $r<n$ such that $\sigma^r(\bb')\prec\bb'$, which means
$
(v_{r+1}\cdots v_n v_1\cdots v_r)  \prec (v_1\cdots v_n).
$ and $v_{r+1}=2$.
If $v_{r+1}\cdots v_n \prec v_1\cdots v_{n-r}$, then comparing the
original sequence $\bb$ at the same shift gives $\sigma^r(\bb)\prec\bb$,
contradicting the minimality of $n$. Hence we must have
$v_{r+1}\cdots v_n = v_1\cdots v_{n-r}$, and we split the proof into
two cases based on the parity of $r$. Here we only give the proof for odd
$n$; the case of even  $n$ is entirely analogous.

\emph{Case 1: $r$ is even.} Then $n-r$ is odd (since $n$ is odd).
By the equality $v_{r+1}\cdots v_n = v_1\cdots v_{n-r}$ and the parity
condition, we obtain $v_1\cdots v_r \succ v_{n-r+1}\cdots v_n$.
If $v_{n-r+1}=2$, then $\sigma^{n-r}(\bb)\prec\bb$, contradicting the
minimality of $n$. If $v_{n-r+1}=1$, then using the structure of $\bb$ and $v_{n+1}=2$,
we get $\sigma^r(\bb)=(v_{r+1}\cdots v_{n}v_{n+1}\cdots)\prec\bb=(v_1\cdots v_{n-r}v_{n-r+1}\cdots)$, again a contradiction.

\emph{Case 2: $r$ is odd.} Then $n-r$ is even, giving
$v_1\cdots v_r \prec v_{n-r+1}\cdots v_n$. By the assumption that $\sigma^n(\bb) \prec \bb$, we have $v_{n+1}\cdots v_{n+r}
\preceq v_1\cdots v_r$.
Consequently,
\[
v_{r+1}\cdots v_n v_{n+1}\cdots v_{n+r}
\preceq v_1\cdots v_{n-r}v_1\cdots v_r
\prec v_1\cdots v_{n-r}v_{n-r+1}\cdots v_n,
\]
which yields $\sigma^r(\bb)\prec\bb$, contradicting the minimality of $n$. Thus $\bb'$ must be self-admissible.

\par (2) Let $n$ be even. We first prove that $\bb'\succ\bb$.
Write $\bb = v_1\cdots v_n \sigma^n(\bb)$. Since $\bb' = (v_1\cdots v_n)^\infty$,
both sequences agree on the first $n$ symbols. Let $k\ge 1$ be the smallest
integer such that the $k$-th blocks of length $n$ differ, i.e.,
\[
v_{(k-1)n+1}\cdots v_{kn} \neq v_1\cdots v_n
\quad\text{for } \bb' \text{ and } \bb.
\]
Such a $k$ exists because $\bb'\neq\bb$. The condition
$\sigma^n(\bb)\prec\bb$ together with the minimality of $k$ implies that
the $k$-th block of $\bb$ is strictly smaller than the corresponding block
of $\bb'$ in the alternating order. Since $n$ is even, comparing sequences
block-by-block is compatible with the alternating lexicographic order,
and we obtain $\bb'\succ\bb$.
\par Next we prove that $\Omega(\ba, \bb) = \Omega(\ba, \bb')$ when $n$ is even.
Since $\bb'\succ\bb$, clearly $\Omega(\ba,\bb')\subseteq\Omega(\ba,\bb)$.
For the reverse inclusion, suppose that
$w\in\Omega(\ba,\bb)\setminus\Omega(\ba,\bb')$. Let $k\geq0$ be the
first time such that
$
u:=\sigma^k(w)\in[\bb,\bb').
$
Put
\[
p:=v_1\cdots v_n,\qquad T:=\sigma^n,\qquad
q:=T(\bb)=\sigma^n(\bb).
\]
Since $n$ is even, deleting the common prefix $p$ preserves the
alternating order. Moreover, $q\prec\bb$ and $T(\bb')=\bb'$. Every
sequence in $[\bb,\bb')$ has the prefix $p$, and hence
\[
q=T(\bb)\preceq T(u)\prec T(\bb')=\bb'.
\]

We claim that some iterate $T^j(u)$ belongs to $(\ba,\bb)$. If
$T(u)\prec\bb$, then $T(u)\succeq q$. Since $q$ starts with $2$,
$T(u)$ also starts with $2$, whereas $\ba$ starts with $1$. Therefore
$
\ba\prec T(u)\prec\bb.
$
If $T(u)\succeq\bb$, then $T(u)\in[\bb,\bb')$, and the same argument
can be repeated.

Suppose that no iterate enters $(\ba,\bb)$. Then
$
T^j(u)\in[\bb,\bb')$ for every $j\geq0.
$
Each $T^j(u)$ begins with $p$. Thus every block of $u$ of length $n$ is
equal to $p$, and consequently
$
u=p^\infty=\bb',
$
which contradicts $u\in[\bb,\bb')$. Hence there exists $j\geq1$ such
that $T^j(u)\in(\ba,\bb)$. It follows that
$
\sigma^{k+jn}(w)=T^j(u)\in(\ba,\bb),
$
contradicting $w\in\Omega(\ba,\bb)$. Therefore
$
\Omega(\ba,\bb)\subseteq\Omega(\ba,\bb').
$
Together with the inclusion proved above, this gives
$
\Omega(\ba,\bb)=\Omega(\ba,\bb').
$
\par (3) Let $n$ be odd.  As in the proof of (2), we verify that
$\bb'\prec\bb$, and hence
$\Omega(\ba,\bb)\subseteq\Omega(\ba,\bb')$.  Take any
\(w\in\Omega(\ba,\bb')\setminus\Omega(\ba,\bb)\). There exists a minimal $k\ge 0$
such that $\sigma^k(w)\in[\bb',\bb)$, and set $u=\sigma^k(w)$. If
$u=\bb'$, then $w$ is a preimage of $\bb'$. Otherwise,
$\bb'\prec u\prec\bb$. As $\bb,\bb'$ share the prefix $v_1\cdots v_n$ and
$u$ lies strictly between them, $u$ also begins with $v_1\cdots v_n$.
Since $n$ is odd, $\sigma^n$ reverses the alternating order on sequences
with a common prefix of length $n$. Applying $\sigma^n$ to
$\bb'\prec u\prec\bb$ yields
\[
\bb'=\sigma^n(\bb')\succ\sigma^n(u)\succ\sigma^n(\bb),
\]
which indicates that
$\sigma^n(u)\in(\ba,\bb')$. Hence $\sigma^{k+n}(w)\in(\ba,\bb')$,
and $w\notin\Omega(\ba,\bb')$.
\par Thus every point in $\Omega(\ba,\bb')\setminus\Omega(\ba,\bb)$ is a preimage
of $\bb'$ that belongs to $\Omega(\ba,\bb')$.
In particular, this set is countable.
\end{proof}

\begin{remark}\label{rerere1}
\begin{enumerate}
\item In Lemma~\ref{lelele11}, $n$ is the first time that a shift violates
self-admissibility; $\bb$ need not be periodic.
The parity conclusions concern this first violating time.
    \item By Lemma~\ref{lelele11}, the two survivor spaces differ by at most
    a countable set. Lemma~\ref{lem:countable-diff-entropy} therefore gives
    $h_{\mathrm{top}}(\Omega(\ba,\bb'))=h_{\mathrm{top}}(\Omega(\ba,\bb)).$
    \item  If $\ba\notin\Omega(\ba,2(21)^\infty)$ or
    $\bb\notin\Omega(1(12)^\infty,\bb)$, we can  replace them by
    self-admissible sequences $\ba'$ and $\bb'$ via Lemma~\ref{lelele11}
    (for the right endpoint) and its symmetric counterpart (for the left
    endpoint). These replacements preserve entropy by
    Lemma~\ref{lem:countable-diff-entropy}, so they may be used in entropy
    calculations. An odd adjustment can change the survivor space by a
    nonempty countable set; for exact survivor-space plateaux and
    bifurcation sets, the original boundaries and the stated
    admissibility hypotheses must therefore be retained.
\end{enumerate}
\end{remark}

\par  The following
lemma deals with the case that $(\ba,\bb)$ is not weak-admissible.

\begin{lemma}\label{lem:self-admissible-adjust}
Let $\ba\in1\{1,2\}^{\mathbb N}$ and
$\bb=(v_1v_2\cdots)\in2\{1,2\}^{\mathbb N}$ be self-admissible.
Suppose that there exists a least $n\ge1$ such that
$\sigma^n(\bb)\succ\ba$ and $v_{n+1}=1$, and put
$p=v_1\cdots v_n$ and $\bv'=p^\infty$. Then $\bv'$ is self-admissible.
Moreover, the following statements hold.
\begin{enumerate}
\item\label{it:self-adjust-even} If $n$ is even, then $\bv'\succ\bb$ and
$
\Omega(\ba,\bv')=\Omega(\ba,\bb).
$
\item\label{it:self-adjust-odd} If $n$ is odd, then $\bv'\prec\bb$,
$\Omega(\ba,\bb)\subseteq\Omega(\ba,\bv')$, and
\[
\Omega(\ba,\bv')\setminus\Omega(\ba,\bb)
=\bigl\{\omega\in\Omega(\ba,\bv'):\
\sigma^k(\omega)=\bv'\text{ for some }k\ge0\bigr\}.
\]
\end{enumerate}
Consequently, $\Omega(\ba,\bv')\triangle\Omega(\ba,\bb)$ is at most
countable in either case.
\end{lemma}

\begin{proof}
The self-admissibility of $\bv'$ follows from the same periodization
argument as in Lemma~\ref{lelele11}; we record the short adaptation.
Assume that $\bv'$ is not self-admissible and choose the least
$1\le r<n$ such that $\sigma^r(\bv')$ starts with $2$ and
$\sigma^r(\bv')\prec\bv'$. If the comparison is decided before the
common block $v_1\cdots v_{n-r}$ is exhausted, the same comparison for
$\bb$ gives $\sigma^r(\bb)\prec\bb$, contradicting the
self-admissibility of $\bb$. Otherwise,
$v_{r+1}\cdots v_n=v_1\cdots v_{n-r}$. If $n-r$ is odd, the remaining
comparison gives $\sigma^{n-r}(\bb)\prec\bb$ when
$v_{n-r+1}=2$. When $v_{n-r+1}=1$, the assumptions give
\[
\sigma^r(\bb)\prec v_{r+1}\cdots v_n\ba
=v_1\cdots v_{n-r}\ba
\preceq v_1\cdots v_{n-r}\sigma^{n-r}(\bb)=\bb,
\]
where $\sigma^{n-r}(\bb)\preceq\ba$ follows from the minimality of $n$.
If $n-r$ is even, the comparison forces $v_{n-r+1}=2$ and directly gives
$\sigma^r(\bb)\prec\bb$.
Every alternative contradicts the self-admissibility of $\bb$.
Therefore $\bv'$ is self-admissible.

For the survivor spaces, put $\mathbf q=\sigma^n(\bb)$. Then
\[
\bb=p\mathbf q,\qquad \bv'=p\bv',\qquad
\ba\prec\mathbf q\prec\bv',
\]
where the last inequality holds because $\mathbf q$ starts with $1$
and $\bv'$ starts with $2$. Deleting the common prefix $p$ preserves
the alternating order when $n$ is even and reverses it when $n$ is odd.
Thus $\bv'\succ\bb$ in the even case and $\bv'\prec\bb$ in the odd case.

\par\emph{Even $n$.} Monotonicity gives
$\Omega(\ba,\bv')\subseteq\Omega(\ba,\bb)$. If
$\omega\in\Omega(\ba,\bb)\setminus\Omega(\ba,\bv')$, some shift
$u$ lies in $[\bb,\bv')$. Hence $u$ begins with $p$ and
$\mathbf q\preceq\sigma^n(u)\prec\bv'$. If
$\sigma^n(u)\prec\bb$, then $\sigma^n(u)\in(\ba,\bb)$, since
$\mathbf q\succ\ba$, contradicting survival. Otherwise
$\sigma^n(u)\in[\bb,\bv')$, and iteration gives the same conclusion
for every $n$-shift of $u$. Every aligned length-$n$ block of $u$ is
therefore $p$, so $u=p^\infty=\bv'$, contradicting $u\prec\bv'$. Thus
$\Omega(\ba,\bv')=\Omega(\ba,\bb)$.

\par\emph{Odd $n$.} Monotonicity gives
$\Omega(\ba,\bb)\subseteq\Omega(\ba,\bv')$. If
$\omega\in\Omega(\ba,\bv')\setminus\Omega(\ba,\bb)$, some shift
$u$ lies in $[\bv',\bb)$. Unless $u=\bv'$, deleting the common
odd-length prefix gives
$\mathbf q\prec\sigma^n(u)\prec\bv'$, hence
$\sigma^n(u)\in(\ba,\bv')$, a contradiction. Thus every such $\omega$
is a shift-preimage of $\bv'$. Conversely, every surviving
shift-preimage of $\bv'$ fails for $(\ba,\bb)$ because
$\ba\prec\bv'\prec\bb$. Therefore
\[
\Omega(\ba,\bv')\setminus\Omega(\ba,\bb)
=\bigl\{\omega\in\Omega(\ba,\bv'):
\sigma^k(\omega)=\bv'\text{ for some }k\ge0\bigr\},
\]
and this set is at most countable.
\end{proof}

\begin{lemma}\label{newlemma1}
Let $\ba$ and $\bb$ be self-admissible.
Then there exists a weak-admissible pair $(\bu,\bv)$ such that
$\Omega(\bu,\bv)\triangle \Omega(\ba,\bb)$ is at most countable.
\end{lemma}
\begin{proof}
Start with $(\ba_0,\bb_0)=(\ba,\bb)$. Whenever the current right
boundary fails to survive, choose its least violating time $n$.
Self-admissibility forces $\sigma^n(\bb_j)$ to start with $1$, so
Lemma~\ref{lem:self-admissible-adjust} replaces the pair by
\[
\ba_{j+1}=\ba_j,\qquad \bb_{j+1}=(\bb_j|_n)^\infty.
\]
Repeat with the left boundary fixed until the right boundary survives.
If the left boundary still fails to survive, adjust it by the symmetric
lemma, keeping the right one fixed and returning to right adjustment
whenever it ceases to survive. Stop when both boundaries survive.

Each adjustment preserves self-admissibility and changes the survivor
space by at most a countable set. The procedure is finite: the first
adjustment of either boundary makes it periodic, and a subsequent
adjustment of a boundary of least period $q$ occurs at a least violating
time $1\leq r<q$. Otherwise periodicity gives an earlier violation
modulo $q$; a multiple of $q$ returns the boundary itself and cannot
be a violating time. Its new period is at most $r<q$, so each boundary can be
adjusted only finitely often, even when the adjustments alternate.
The final pair $(\bu,\bv)$ is weak-admissible, and the finite union
of the successive symmetric differences is countable. If every
adjustment occurs at an even first-violating time,
Lemma~\ref{lem:self-admissible-adjust}\textup{(\ref{it:self-adjust-even})}
and its symmetric counterpart give the stronger equality
$\Omega(\bu,\bv)=\Omega(\ba,\bb)$.
\end{proof}

\begin{remark}\label{rererere2}
\begin{enumerate}
\item Lemma~\ref{newlemma1} gives at most countable symmetric difference
between $\Omega(\bu,\bv)$ and $\Omega(\ba,\bb)$. By
Lemma~\ref{lem:countable-diff-entropy},
$h_{\mathrm{top}}(\Omega(\bu,\bv))=h_{\mathrm{top}}(\Omega(\ba,\bb))$.
\item Each boundary moves monotonically in the alternating order
through its successive adjustments; neither boundary reverses direction.
\end{enumerate}
\end{remark}

\vspace{1em}

\begin{proof}[Proof of Theorem~\ref{thm:dlm-hole}]
For the first assertion, the hypotheses say precisely that the pair
$(\ba,\bb)$ is weak-admissible and has the strict jump ordering. Lemma \ref{lem:dlm-realization} stated above therefore gives a map $g$ with
$
\Omega(g)=\Omega(\ba,\bb).
$

For the converse, let $(k_-,k_+)$ be the kneading pair of $g$.
Monotonicity implies that every realized one-sided itinerary beginning
with $1$ is at most $k_-$, and one beginning with $2$ is at least $k_+$.
By the one-sided shift identities, these bounds hold for every shift, giving
$k_-,k_+\in\Omega(k_-,k_+)$ and
$\Omega(g)\subseteq\Omega(k_-,k_+)$.

For the reverse inclusion, the extremal sequences
$e_0=(12)^\infty$ and $e_1=(21)^\infty$ are realized by $0$ and $1$.
For $e_0\prec\omega\prec e_1$ in $\Omega(k_-,k_+)$, set
\[
x=\sup\{y\in[0,1]\setminus C(g):\tau_g(y)\prec\omega\}.
\]
Stability of finite endpoint prefixes gives $0<x<1$, and monotonicity
and the one-sided limits yield
\[
\tau_g(x-)\preceq\omega\preceq\tau_g(x+).
\]
If both inequalities were strict, $x$ would be precritical.
At its first arrival time $r$ at $c$, the two bounding itineraries
have a common prefix of length $r$ and tails $k_-,k_+$.
Thus $\omega$ shares that prefix; deleting it, with order reversed
when $r$ is odd, gives $\sigma^r(\omega)\in(k_-,k_+)$, a contradiction.
Hence $\omega$ is a one-sided itinerary of $x$, proving
$\Omega(g)=\Omega(k_-,k_+)$.

For the geometric assertion, set
$a=\pi_{-2}(k_-)$ and $b=\pi_{-2}(k_+)$, and assume
$0<a<1/2<b<1$. The critical sequences
$\ell_0=1(12)^\infty$ and $r_0=2(21)^\infty$ of $T_{-2}$ then satisfy
\[
k_-\prec\ell_0\prec r_0\prec k_+.
\]
Weak admissibility excludes these two sequences from every shift of
$k_-$ and $k_+$. Since each critical or precritical negative-doubling
sequence has such a shift, $a,b\notin C_{-2}$. Their unique itineraries
therefore give $\tau_{-2}(a-)=k_-$ and $\tau_{-2}(b+)=k_+$,
which, together with the symbolic equality, proves the hole representation.
For the canonical pair $(\ell_0,r_0)$, the same symbolic equality gives
the representation by $T_{-2}$ with the empty hole.
\end{proof}

Numerical addresses alone do not ensure the geometric conclusion of
Theorem~\ref{thm:dlm-hole}.
For example, $(k_-,k_+)=(11r_0,r_0)$ is weak-admissible and satisfies
$\sigma(k_-)\prec\sigma(k_+)$, so Lemma~\ref{lem:dlm-realization}
realizes it in $\mathrm{DLM}$. Its negative-doubling addresses are
$a=3/8$ and $b=1/2$, but
\[
\tau_{-2}\!\left(\frac38-\right)=11\ell_0\ne11r_0=k_-.
\]
Thus these numerical addresses do not recover the prescribed lower
boundary sequence: the symbolic equality remains valid, but the standard
hole $(3/8,1/2)$ does not give the claimed geometric representation.

By Lemmas~\ref{newlemma1} and~\ref{lem:dlm-realization}, a
self-admissible pair $(\ba,\bb)$ can be adjusted to $(\bu,\bv)$ with
countable survivor-space difference. If the adjusted pair also satisfies
$\sigma(\bu)\prec\sigma(\bv)$, it is realized by $g\in\mathrm{DLM}$,
and $\Omega(g)\triangle\Omega(\ba,\bb)$ is at most countable.

 \vspace{1em}
The following examples illustrate the parity distinction.
\begin{example}\label{exam1}
Consider the negative doubling map with a hole $(a,b)$.
\begin{enumerate}
\item Let $a=1/2$ and $\bb=(21112112)^\infty$. By Lemma~\ref{lelele11},
  there exists an even integer $n=4$ such that $\sigma^4(\bb)\prec\bb$
  (and $\sigma^4(\bb)$ begins with the symbol $2$). Hence we obtain
  $\bb'=(2111)^\infty\succ\bb$ and
  $\Omega(\ba,\bb)=\Omega(\ba,\bb')$.

\item Let $a=1/2$ and $\bb=(2211222111)^\infty$. By Lemma~\ref{lelele11},
  there exists an odd integer $n=5$ such that $\sigma^5(\bb)\prec\bb$
  (and $\sigma^5(\bb)$ begins with the symbol $2$). Hence we obtain
  $\bb'=(22112)^\infty\prec\bb$ and
  $\Omega(\ba,\bb')\setminus\Omega(\ba,\bb)$ is countable.

\item  Let $\ba=(112)^\infty$ and $\bb=(2211121)^\infty$.
  Both sequences are self-admissible, but $(\ba,\bb)$ is not
  weak-admissible. Indeed, $\sigma(\bb)$ starts with $2$,
  $\sigma^2(\bb)\prec\ba$, and $\sigma^3(\bb)\succ\ba$, with
  $\sigma^3(\bb)$ starting with $1$. Thus the first violating time
  in Lemma~\ref{lem:self-admissible-adjust} is $n=3$. By its odd case,
  Lemma~\ref{lem:self-admissible-adjust}\textup{(\ref{it:self-adjust-odd})},
  the adjusted boundary is $\bb'=(221)^\infty\prec\bb$, and
  \[
  \Omega(\ba,\bb')\setminus\Omega(\ba,\bb)
  =\bigl\{\omega\in\Omega(\ba,\bb'):
    \sigma^k(\omega)=\bb'\text{ for some }k\ge0\bigr\}.
  \]
  The pair $(\ba,\bb')$ is weak-admissible.
  In particular, $\bb'$ itself belongs to the displayed difference.
  Hence $\Omega(\ba,\bb)\subsetneq\Omega(\ba,\bb')$, and the
  difference is at most countable.
\end{enumerate}
\end{example}

\subsection{Proof of Theorem \ref{plateauxin}}\label{ss4.2}
We retain the symbolic plateau $P(\bb)$ among all sequences beginning
with $2$; geometric right-endpoint parameters remain in $[1/2,1)$.
\begin{lemma}\label{lee21}
 Let $a=1/2$ be fixed and let $\bb=(v_1v_2\cdots)$ be not self-admissible. Suppose $n$ is the
minimal integer such that $\sigma^n(\bb)\prec\bb$ with $v_{n+1}=2$.
Define $\bb'=(v_1\cdots v_n)^\infty$.
\begin{enumerate}
\item If $n$ is even, then $P(\bb)=[\bb_L,\bb']$, where
  $\bb_L=v_1\cdots v_n\,2(21)^\infty$.
  \item  If $n$ is odd, then $P(\bb)=(\bb',\bb_R]$, where
  $\bb_R=v_1\cdots v_n\,1(12)^\infty$.
\end{enumerate}
\end{lemma}

\begin{proof}
Put $p=v_1\cdots v_n$, $\ell_0=1(12)^\infty$ and
$r_0=2(21)^\infty$, so $\ba=\ell_0$. Right boundary parameters
always use upper sequences. A self-admissible right boundary is the least
sequence beginning with $2$ in its survivor space, which therefore
determines it uniquely.
We first establish two observations for a non-self-admissible
right boundary $s=(s_i)_{i\geq1}\in2\{1,2\}^{\mathbb N}$.
Let $r$ be the least positive integer such that $s_{r+1}=2$ and
$\sigma^r(s)\prec s$. These observations will be used in both parity
cases below.

If $r$ is even, it is the least even period of
$z=(s|_r)^\infty$. Indeed, $z\succ s$ by Lemma~\ref{lelele11}.
A smaller even period $e<r$ would give $s_{e+1}=2$ and
$\sigma^e(s)\prec s$: the first mismatch between $s$ and $z$ moves
$e$ places to the left without changing parity, contradicting minimality.

If $r$ is odd, set $R=(s|_r)\ell_0$.
Then $s\preceq(s|_r)r_0\prec R$ and
$\Omega(\ell_0,s)=\Omega(\ell_0,R)$: a shift $u\in[s,R)$ satisfies
$\ell_0\prec\sigma^r(u)\preceq\sigma^r(s)\prec s$.
To check self-admissibility of $R$, consider $1\leq j<r$ with
$s_{j+1}=2$, and put $h=r-j$, $q=\sigma^r(s)$. If the length-$h$
prefixes of $\sigma^j(s)$ and $s$ differ, their comparison is unchanged
for $R$. Otherwise minimality of $r$ gives
\[
\begin{cases}
q\succeq\sigma^h(s),\quad s_{h+1}=1,&h\text{ even},\\
q\preceq\sigma^h(s),\quad s_{h+1}=2,&h\text{ odd}.
\end{cases}
\]
Comparing $\ell_0$ with $\sigma^h(R)$ and restoring the common prefix
gives $\sigma^j(R)\succeq R$. All later shifts are $\ell_0$,
$(12)^\infty$ or $(21)^\infty$. Thus $R$ is self-admissible and
nonperiodic.

\par\noindent\textup{(1)} Let $n$ be even. Lemma~\ref{lelele11} gives
\[
\bb_L=pr_0\preceq\bb\prec\bb'=p^\infty,\qquad
\Omega(\ell_0,\bb)=\Omega(\ell_0,\bb'),\qquad
\bb'\in\Omega(\ell_0,\bb').
\]
For $\epsilon\in[\bb_L,\bb']$, monotonicity gives one inclusion.
If a sequence survived for $\epsilon$ but not for $\bb'$, some shift
$u$ would lie in $[\epsilon,\bb')$. Deleting its even prefix $p$ gives
$r_0\preceq\sigma^n(u)\prec\bb'$, and survival forces
$\sigma^n(u)\succeq\epsilon$. Repetition forces $u=p^\infty$,
a contradiction. Hence $[\bb_L,\bb']\subseteq P(\bb)$.

Suppose $w\prec\bb_L$ gave the same survivor space. It cannot be
self-admissible, since it would then belong to $\Omega(\ell_0,\bb')$.
Let $m$ be its first violating time. If $m$ is even,
Lemma~\ref{lelele11} and  uniqueness of the self-admissible boundary  give
$(w|_m)^\infty=\bb'$. Both $m$ and $n$ are its least even period,
so $m=n$. Hence $w$ has prefix $p$ and $w_{n+1}=2$, giving
$w\succeq pr_0=\bb_L$, a contradiction. If $m$ is odd, the odd
adjustment above yields the nonperiodic self-admissible boundary
$(w|_m)\ell_0$ with the same survivor space, which would have to equal
the periodic boundary $\bb'$, again impossible. Finally, $w\succ\bb'$
excludes $\bb'$ itself from the survivor space. Therefore
$P(\bb)=[\bb_L,\bb']$.

\par\noindent\textup{(2)} Let $n$ be odd and put $\bb_R=p\ell_0$,
the upper endpoint sequence. The odd adjustment gives
$\Omega(\ell_0,\bb)=\Omega(\ell_0,\bb_R)$ and
$\bb_R\in\Omega(\ell_0,\bb_R)$, while
Lemma~\ref{lelele11} gives $\bb'\prec\bb$ and self-admissibility of
$\bb'$. Hence monotonicity gives
\[
\Omega(\ell_0,\bb)=\Omega(\ell_0,\bb_R)
\subsetneq\Omega(\ell_0,\bb').
\]
The inclusion is strict because $\bb'$ survives for $\bb'$ but lies
in the forbidden interval $(\ell_0,\bb)$.
For $\epsilon\in(\bb',\bb_R]$, a shift
$u\in[\epsilon,\bb_R)$ would satisfy
$\ell_0\prec\sigma^n(u)\prec\bb'\prec\epsilon$. Thus no
sequence survives for $\epsilon$ but not for $\bb_R$, and monotonicity
gives $(\bb',\bb_R]\subseteq P(\bb)$. For $w\preceq\bb'$,
the sequence $\bb'$ survives for $w$ but not for $\bb$; for
$w\succ\bb_R$, the sequence $\bb_R$ survives for $\bb$ but not
for $w$. Therefore $P(\bb)=(\bb',\bb_R]$.
\end{proof}

Lemma~\ref{lee21} describes the plateau in the special case
$a=1/2$ when $\bb$ is not self-admissible.
We now consider a general fixed left boundary $\ba$ and study
$P(\bb)$ under the assumption that $(\ba,\bb)$ is weak-admissible.

\begin{proposition}\label{prop:boundary-return-plateaux}
Fix $\ba$ and suppose that $(\ba,\bb)$ is weak-admissible.
Define
\[
n:=\min\{2j:j\geq1,\ \sigma^{2j}(\bb)=\bb\},
\qquad
m:=\min\{2j+1:j\geq0,\ \sigma^{2j+1}(\bb)=\ba\},
\]
with  $\min\varnothing=+\infty$, and put
$k:=\min\{n,m\}$.
If $k<+\infty$, put $w=\bb|_k$ and define
\[
L=
\begin{cases}
w\ba,&k\text{ even},\\
w^\infty,&k\text{ odd}.
\end{cases}
\]
Then $L\in\Omega(\ba,L)$ and $P(\bb)=(L,\bb]$.
\end{proposition}

\begin{proof}
Assume $k<+\infty$ and put $S=\Omega(\ba,\bb)$. Since $\bb\in S$, every right boundary
$\eta\succ\bb$ excludes $\bb$, so $P(\bb)$ cannot extend to the
right of $\bb$.

\par \emph{1. $L\in\Omega(\ba,L)$ holds for both cases.}
Since $\ba\prec L\prec\bb$, monotonicity gives
$S\subseteq\Omega(\ba,L)$. If $k$ is even, then
$\sigma^k(L)=\ba$ and all later shifts are shifts of $\ba$;
if $k$ is odd, then $\sigma^k(L)=L$. It remains to check
$1\leq j<k$. Let $w=\bb|_k=w_1\cdots w_k$.
For such $j$, the sequences $\sigma^j(L)$ and
$\sigma^j(\bb)$ share the prefix $w_{j+1}\cdots w_k$. Define
\[
J_j=
\begin{cases}
\bigl(\sigma^j(L),\sigma^j(\bb)\bigr],&j\text{ even},\\
\bigl[\sigma^j(\bb),\sigma^j(L)\bigr),&j\text{ odd}.
\end{cases}
\]
The common-prefix cylinder implies that every
$x\in J_j$ has this prefix. Deleting it gives
\[
\sigma^{k-j}(x)\in
\begin{cases}
(\ba,\bb],&k\text{ even},\\
[\ba,L),&k\text{ odd}.
\end{cases}
\]
If $x\in J_j\cap S$, shift-invariance and avoidance of $(\ba,\bb)$
force $\sigma^{k-j}(x)=\sigma^k(\bb)$. On the other hand, both
$x$ and $\sigma^j(\bb)$ begin with the common word
$w_{j+1}\cdots w_k$. After this common prefix is deleted, the tail
of $x$ is $\sigma^{k-j}(x)=\sigma^k(\bb)$, while the tail of
$\sigma^j(\bb)$ is also $\sigma^k(\bb)$. Hence the two sequences have
the same prefix and the same tail, so
\[
x=(w_{j+1}\cdots w_k)\sigma^k(\bb)=\sigma^j(\bb).
\]
Therefore $J_j\cap S=\{\sigma^j(\bb)\}$.

We now use this claim to exclude the possibility that \(\sigma^j(L)\in(\ba,L)\). Suppose first
that $j$ is even and $\sigma^j(L)\in(\ba,L)$. Put
$q_j=\sigma^j(\bb)$. Since $q_j\in S$ and
$q_j\succeq\sigma^j(L)\succ\ba$, it cannot lie in the hole
$(\ba,\bb)$; hence $q_j\succeq\bb$. Thus $\bb\in J_j\cap S$, and the
identity just proved gives $q_j=\bb$, i.e.
$\sigma^j(\bb)=\bb$. This is an even self-return with $j<k\le n$,
contradicting the definition of $n$.

If $j$ is odd and $\sigma^j(L)\in(\ba,L)$, then
$q_j=\sigma^j(\bb)\prec\sigma^j(L)\prec\bb$. Since $q_j\in S$, it
cannot lie in $(\ba,\bb)$, so $q_j\preceq\ba$. Thus
$\ba\in J_j\cap S$, and the identity above gives $q_j=\ba$, i.e.
$\sigma^j(\bb)=\ba$. This is an odd arrival with $j<k\le m$,
contradicting the definition of $m$. Hence
$\sigma^j(L)\notin(\ba,L)$ for every $1\le j<k$, and therefore
$L\in\Omega(\ba,L)$.

\par \emph{2. Even self-return and the plateau.}
Assume $k=n<+\infty$, so that $n$ is even and
$\sigma^n(\bb)=\bb$. Then $w=\bb|_n$, $L=w\ba$, and
$\sigma^n(L)=\ba$. Since $\ba\prec L\prec\bb$, monotonicity gives
$S\subseteq\Omega(\ba,\epsilon)$ for every $\epsilon\preceq\bb$.
The endpoint $\epsilon=\bb$ belongs to $P(\bb)$ by definition. It
remains to consider $L\prec\epsilon\prec\bb$. Suppose, for a
contradiction, that
$\omega\in\Omega(\ba,\epsilon)\setminus S$. Choose a shift
$u=\sigma^r(\omega)\in[\epsilon,\bb)$. Since $u$ begins with the
common even prefix $w$,
\[
\ba\prec\sigma^n(\epsilon)\preceq\sigma^n(u)\prec
\sigma^n(\bb)=\bb.
\]
If $\sigma^n(u)\prec\epsilon$, then
$\sigma^n(u)\in(\ba,\epsilon)$, contradicting
$\omega\in\Omega(\ba,\epsilon)$. Hence
$\sigma^n(u)\in[\epsilon,\bb)$. Repeating this argument gives
$\sigma^{jn}(u)\in[\epsilon,\bb)$ for every $j\ge0$, so every aligned
length-$n$ block of $u$ equals $w$. Thus $u=w^\infty=\bb$, contradicting
$u\prec\bb$. Therefore $\Omega(\ba,\epsilon)=S$, and
$(L,\bb]\subseteq P(\bb)$.
\par \emph{3. Odd arrival at the left boundary and the plateau.}
Assume that $k=m<+\infty$, so that $m$ is odd and
$\sigma^m(\bb)=\ba$. Then $w=\bb|_m$, $L=w^\infty$, and
$\sigma^m(L)=L$. If $L\prec\epsilon\prec\bb$ and a sequence
survived for $\epsilon$ but not for $\bb$, a shift $u\in[\epsilon,\bb)$
would, after deleting the common odd prefix $w$, satisfy
$\ba\prec\sigma^m(u)\prec L\prec\epsilon$, contradicting survival.
Therefore $(L,\bb]\subseteq P(\bb)$.

\par \emph{4. Exact plateau for both cases.} Finally, for every right boundary $\eta\preceq L$, the sequence
$L$ survives for $(\ba,\eta)$ but not for $(\ba,\bb)$, since
$\ba\prec L\prec\bb$. Together with the interval inclusion
already proved and the exclusion of $\eta\succ\bb$,  this gives
$P(\bb)=(L,\bb].$
\end{proof}

\begin{example}\label{ex:two-boundary-return-plateaux}
Fix $\ba=(112)^\infty$ and consider
\[
\bb_1=(221)^\infty,\qquad
\bb_2=221(112)^\infty.
\]
Both pairs $(\ba,\bb_i)$ are weak-admissible and illustrate the two
finite-return cases of Proposition~\ref{prop:boundary-return-plateaux}.
\begin{enumerate}
\item\label{it:boundary-return-even-example}
For $\bb_1=(221)^\infty$, one has
$\sigma(\bb_1)=(212)^\infty\succ\bb_1$ and
$\sigma^2(\bb_1)=(122)^\infty\prec\ba$.
Its least period is $3$, so the least even self-return is $n=6$;
no shift equals $\ba$, giving $m=+\infty$. Thus $w=221221$ and
\[
P(\bb_1)=\bigl(221221(112)^\infty,\,(221)^\infty\bigr].
\]
Although the least period is odd, the relevant return time is even,
and the left endpoint is $L_1=w\ba$.

\item For $\bb_2=221(112)^\infty$, one has
$\sigma(\bb_2)=21(112)^\infty\succ\bb_2$,
$\sigma^2(\bb_2)=1(112)^\infty\prec\ba$, and
$\sigma^3(\bb_2)=\ba$. The block $22$ occurs only initially, so
$\bb_2$ is not periodic. Hence $n=+\infty$, $m=3$, $w=221$, and
\[
P(\bb_2)=\bigl((221)^\infty,\,221(112)^\infty\bigr].
\]
\end{enumerate}
Thus $L_2=(221)^\infty=\bb_1$ is the included right endpoint of the first plateau
and the excluded left endpoint of the second. Moreover,
$\bb_1\in\Omega(\ba,\bb_1)\setminus\Omega(\ba,\bb_2)$, so
$\Omega(\ba,\bb_2)\subsetneq\Omega(\ba,\bb_1)$.
\end{example}

We next treat the complementary case $k=+\infty$.

\begin{proposition}\label{prop:singleton-plateau}
Suppose that $(\ba,\bb)$ is weak-admissible, and let $k$ be as defined
in Proposition~\ref{prop:boundary-return-plateaux}. Then
\[
P(\bb)=\{\bb\}
\quad\Longleftrightarrow\quad k=+\infty.
\]
\end{proposition}
\begin{proof}
If $k<+\infty$, Proposition~\ref{prop:boundary-return-plateaux}
gives $P(\bb)=(L,\bb]$, so the plateau is not a singleton. It remains
to consider $k=+\infty$. Put $S=\Omega(\ba,\bb)$ and suppose that
$\epsilon\in P(\bb)$ and $\epsilon\neq\bb$.
Since
$\bb\in S$, every larger right boundary excludes $\bb$; hence the only possibility is
$\epsilon\prec\bb$ and $\sigma(\bb)\prec\sigma(\epsilon)$.
Moreover,
\[
S\cap(\sigma(\bb),\sigma(\epsilon)]=\varnothing,
\]
Indeed, if $\eta\in S\cap(\sigma(\bb),\sigma(\epsilon)]$, then
prefixing by $2$ reverses the alternating order and gives
\[
\epsilon\preceq 2\eta\prec\bb.
\]
Since $\ba\prec\epsilon\prec\bb$, this implies $2\eta\notin S$.
On the other hand, $S\subseteq\Omega(\ba,\epsilon)$ by monotonicity,
$2\eta\succeq\epsilon$ avoids the hole at time zero, and for every
$j\ge1$,
\[
\sigma^j(2\eta)=\sigma^{j-1}(\eta)\in S
\subseteq\Omega(\ba,\epsilon).
\]
Thus $2\eta\in\Omega(\ba,\epsilon)\setminus S$, contradicting
$\Omega(\ba,\epsilon)=S$.
The preceding exclusion shows that $\sigma(\epsilon)\notin S$.
The sequence $(21)^\infty$ is the maximal sequence in the alternating
order; its shifts are $(21)^\infty$ and $(12)^\infty$, so it avoids
$(\ba,\bb)$ and belongs to $S$. Hence
$(21)^\infty\succ\sigma(\epsilon)$, and compactness of $S$ gives a
least element
\[
\nu=\min\bigl(S\cap[\sigma(\epsilon),(21)^\infty]\bigr).
\]
Put $x_0=\sigma(\bb)$ and $y_0=\nu$. Then $x_0,y_0\in S$ and no
point of $S$ lies strictly between them. Write
$S_i=\{\omega\in S:\omega_1=i\}$. Weak-admissibility gives
$\max S_1=\ba$, $\min S_2=\bb$, and
\[
\sigma(S_1)=\{\eta\in S:\eta\succeq\sigma(\ba)\},\qquad
\sigma(S_2)=\{\eta\in S:\eta\preceq\sigma(\bb)\}.
\]
Thus each $\sigma(S_i)$ is an order interval in $S$. For $j\ge0$, write
$x_j=\sigma^j(x_0)$ and $y_j=\sigma^j(y_0)$. As long as
$x_j$ and $y_j$ have the same first symbol, deleting this common
symbol leaves two $S$-points with no $S$-point between them. Let $r$
be the maximal common-prefix length of $x_0$ and $y_0$, and put
$x_r=\sigma^r(x_0)$ and $y_r=\sigma^r(y_0)$. If $r$ is even, then
$x_r\prec y_r$, so the extremal property of $\ba$ and $\bb$ gives
$x_r=\ba$ and $y_r=\bb$. If $r$ is odd, then $x_r\succ y_r$, so
$x_r=\bb$ and $y_r=\ba$. Since
$x_r=\sigma^{r+1}(\bb)$, we obtain
\[
\sigma^{r+1}(\bb)=
\begin{cases}
\ba,&r\text{ even},\\
\bb,&r\text{ odd}.
\end{cases}
\]
This contradicts $k=+\infty$. Hence
$P(\bb)=\{\bb\}$.
\end{proof}

\begin{example}\label{ex:singleton-survivor-plateau}
The pair $\ba=(112)^\infty$, $\bb=22\,1^\infty$ is weak-admissible.
The right boundary is eventually periodic but not periodic, and none
of its odd shifts $2\,1^\infty$ and $1^\infty$ equals $\ba$.
Thus $n=m=k=+\infty$, and Proposition~\ref{prop:singleton-plateau} gives
$P(\bb)=\{22\,1^\infty\}$. An eventually periodic right boundary
can therefore have a singleton plateau.
\end{example}

\vspace{1em}
\noindent\textbf{Proof of Theorem~\ref{plateauxin}.}
\par
Monotonicity makes $P(\bb)$ order-convex: for
$\epsilon_1\preceq\epsilon\preceq\epsilon_2$ with
$\epsilon_1,\epsilon_2\in P(\bb)$,
\[
\Omega(\ba,\epsilon_2)\subseteq\Omega(\ba,\epsilon)
\subseteq\Omega(\ba,\epsilon_1)=\Omega(\ba,\bb).
\]
The outer spaces agree, and maximality follows from the definition of
$P(\bb)$. The finite return times $n,m$ have opposite parity, so exactly one
of $n<m$, $m<n$, or $n=m=+\infty$ holds. The first two cases give
$P(\bb)=(L,\bb]$ and $L\in\Omega(\ba,L)$ by
Proposition~\ref{prop:boundary-return-plateaux}; the last gives
$P(\bb)=\{\bb\}$ by Proposition~\ref{prop:singleton-plateau}.
\hfill$\square$

\subsection{Proof of Theorem \ref{theorem2}}
We now prove the devil's staircase property of the Hausdorff dimension
function.  For fixed $\ba$, we let $b$ vary in $[1/2,1)$
and use the boundary  sequences  defined below. All symbolic
parameter neighbourhoods are relative to this set of boundary  sequences.
We write $E(\ba)$ for the symbolic bifurcation set and $E(a)$
for its geometric counterpart.

\begin{definition}\label{deea}
Let \(1/2\leq b<1\), and write
$
\bb=\tau_{-2}(b+).
$
The {\bf symbolic bifurcation set}
\(E(\ba)\) consists of all such sequences \(\bb\) for which every cylinder
neighbourhood of \(\bb\) contains the boundary sequence \(\bb'\) of some
\(b'\in[1/2,1)\) satisfying
$
\Omega(\ba,\bb')\ne\Omega(\ba,\bb).
$
Equivalently,
\[
\begin{aligned}
E(\ba)=\{\,\bb:\ &\forall n\ge1,\ \exists\,\bb'\ne\bb
\text{ with }\bb'|_n=\bb|_n\\
&\text{and }\Omega(\ba,\bb')\ne\Omega(\ba,\bb)\,\},
\end{aligned}
\]
where both \(\bb\) and \(\bb'\) range over the boundary sequences just defined.
Under the negative-base $(-2)$ expansion, the corresponding geometric
parameter set is
\[
\begin{aligned}
E(a)=\{\,b\in[1/2,1):\ &\forall\delta>0,\ \exists\,b'\in[1/2,1)\\
&\text{with }0<|b-b'|<\delta
\text{ and }S_{-2}(a,b')\ne S_{-2}(a,b)\,\}.
\end{aligned}
\]
\end{definition}

\begin{proposition}\label{prprp4.2}
For fixed $\ba$ and \(1/2\leq b<1\), let
\(\bb=\tau_{-2}(b+)\). Then
\[
\bb\in E(\ba)\quad\Longleftrightarrow\quad
\bb\in\Omega(\ba,\bb).
\]
\end{proposition}

\begin{proof}
We first prove the left inclusion by contraposition.
Suppose that $\bb\notin\Omega(\ba,\bb)$. Choose $n\geq1$ with
$\ba\prec\sigma^n(\bb)\prec\bb$.   Choose $M\geq1$ such that both
comparisons are determined by the first $M$ symbols, and take
$N\geq n+M$. Let
\[
U=\{\xi\in2\{1,2\}^{\mathbb N}:\xi|_N=\bb|_N\}
\]
be the corresponding order-convex cylinder. For every $\xi,\eta\in U$,
we have
\[
\sigma^n(\xi)|_M=\sigma^n(\bb)|_M,
\qquad\eta|_M=\bb|_M.
\]
Consequently, the same strict comparisons hold uniformly:
\[
\ba\prec\sigma^n(\xi)\prec\eta
\qquad(\xi,\eta\in U).
\]
Fix any boundary sequence $\epsilon\in U$ corresponding to a right
boundary in $[1/2,1)$,
and put
$\mathbf{c}=\min\{\bb,\epsilon\}$ and
$\mathbf{d}=\max\{\bb,\epsilon\}$, with respect to the alternating order.
Monotonicity gives
$\Omega(\ba,\mathbf{d})\subseteq\Omega(\ba,\mathbf{c})$. If the inclusion
were strict, choose
$
\omega\in\Omega(\ba,\mathbf{c})\setminus\Omega(\ba,\mathbf{d}).
$
There is then some $k\geq0$ such that
$\sigma^k(\omega)\in(\ba,\mathbf{d})$. Since $\omega$ survives for
$(\ba,\mathbf{c})$, this shift cannot lie in $(\ba,\mathbf{c})$. Hence
\[
u:=\sigma^k(\omega)\in[\mathbf{c},\mathbf{d})\subseteq U,
\]
where the last inclusion follows from order-convexity and
$\mathbf{c},\mathbf{d}\in U$. Applying the uniform comparison to
$\xi=u$ and $\eta=\mathbf{c}$ gives
$
\sigma^{k+n}(\omega)=\sigma^n(u)\in(\ba,\mathbf{c}),
$
contradicting $\omega\in\Omega(\ba,\mathbf{c})$. Thus the survivor
space is constant on all upper-sequence parameters in $U$, so
$\bb\notin E(\ba)$. This proves the left inclusion.

\par  For the right inclusion, suppose that
$\bb=\tau_{-2}(b+)\in\Omega(\ba,\bb)$, where $1/2\le b<1$.
We use the standard monotonicity and right-continuity of the upper-sequence
map $b\mapsto\tau_{-2}(b+)$ on $[1/2,1)$, with the right-completed
convention at $c$: it is strictly increasing in the alternating order
and right-continuous in the finite-prefix topology.
For completeness, for each $N$, the first $N$ symbols are constant
on each component of the complement of the finite set
\[
\bigcup_{j=0}^{N-1}T_{-2}^{-j}(1/2).
\]
At these points the upper sequence selects the prefix from the right;
in particular, $\tau_{-2}(c+)=r_0$.
Thus, for every
$N\ge1$, we can choose
$b'\in(b,1)$ sufficiently close to $b$ (also when $b=c$) so that
$\epsilon:=\tau_{-2}(b'+)$ satisfies
$\epsilon|_N=\bb|_N$ and $\epsilon\succ\bb$.
Then $\bb\in(\ba,\epsilon)$, so
$\bb\notin\Omega(\ba,\epsilon)$ although $\bb\in\Omega(\ba,\bb)$.
Every cylinder neighbourhood therefore contains a parameter with a
different survivor space. Hence $\bb\in E(\ba)$, proving the reverse inclusion.
\end{proof}

The next example distinguishes survivor and entropy plateaux.

\begin{example}\label{newplatex}
Let $\ba=(112)^\infty$ and $\bb=(221)^\infty$, as in
Example~\ref{ex:two-boundary-return-plateaux}\textup{(\ref{it:boundary-return-even-example})}, where the exact survivor plateau $P(\bb)$ was established. We now
exhibit an entropy plateau that strictly contains $P(\bb)$.
First we verify the positive-entropy hypothesis of Lemma~\ref{entropy}.
Let $X$ consist of all infinite concatenations of the blocks $11$ and
$22$, together with their shifts. This is a compact subshift contained
in $\Omega(\ba,\bb)$: a sequence beginning with $1$ starts with $12$,
$111$, or $1122$, and one beginning with $2$ starts with $21$, $222$,
or $2211$; these prefixes give the required alternating inequalities.
There are at least $2^n$ distinct words of length $2n$ in $X$, so
\[
h_{\mathrm{top}}(\Omega(\ba,\bb))
\geq h_{\mathrm{top}}(X)\geq\frac{\log2}{2}>0.
\]
The kneading determinant is
\[
K(z)=\sum_{n=1}^\infty (b_n-a_n)z^{n-1}
     =\frac{1+z-z^2}{1-z^3}.
\]
Its largest negative root is $z_0=(1-\sqrt5)/2$, so Lemma~\ref{entropy} gives
\[
h_{\mathrm{top}}(\Omega(\ba,\bb))
=\log(-1/z_0)
=\log\frac{1+\sqrt5}{2}.
\]
To extend the entropy plateau beyond $P(\bb)$, define
\[
\bb_h^-:=221(221112)^\infty,
\qquad
\bb_h^+:=(221112)^\infty.
\]
A direct check  shows that both pairs
$(\ba,\bb_h^-)$ and $(\ba,\bb_h^+)$ are weak-admissible.
Moreover, a sequence of $X$ beginning with $2$ starts with $21$, $222$,
$221111$, or $22112$, each of which places it above $\bb_h^+$.
The left-boundary comparison is unchanged, and hence
$X\subseteq\Omega(\ba,\bb_h^+)\subseteq\Omega(\ba,\bb_h^-)$.
Both endpoint spaces therefore have positive entropy, and
Lemma~\ref{entropy} applies at both endpoints. Their determinants are
\[
K_{\bb_h^+}(z)=\frac{1+z-z^2}{1-z^6},\qquad
K_{\bb_h^-}(z)=\frac{(1+z-z^2)(1+z^3-z^6)}{1-z^6}.
\]
Both determinants have the same largest negative root $z_0$: the
denominator is nonzero on $(-1,0)$, and the additional factor
$1+z^3-z^6$ is positive on $[z_0,0]$. Thus both endpoint systems
have entropy $\log((1+\sqrt5)/2)$. Monotonicity then gives this same entropy throughout
$[\bb_h^-,\bb_h^+]$, an entropy plateau containing $\bb$.
Moreover, the ordering
\[
\bb_h^-\prec221221(112)^\infty
\prec\bb\prec221(112)^\infty\prec\bb_h^+
\]
and Example~\ref{ex:two-boundary-return-plateaux}\textup{(\ref{it:boundary-return-even-example})}
give $P(\bb)\subsetneq[\bb_h^-,\bb_h^+]$. For example,
$\bb\prec\epsilon\prec221(112)^\infty$ excludes $\bb$ from the survivor
space while preserving entropy. No maximality of the entropy plateau
is asserted.
\end{example}

By Example~\ref{newplatex}, the topological entropy may remain unchanged even
when the survivor set varies.  We therefore introduce the following {\bf
bifurcation set for topological entropy} $B(\ba)$:
\[
\begin{aligned}
B(\ba):=\{\,\bb:\ &\forall n\ge1,\ \exists\,\bb'\ne\bb
\text{ with }\bb'|_n=\bb|_n\\
&\text{and }h_{\mathrm{top}}(\Omega(\ba,\bb'))
\ne h_{\mathrm{top}}(\Omega(\ba,\bb))\,\}.
\end{aligned}
\]
Here \(\bb,\bb'\) range over the boundary sequences of parameters in
\([1/2,1)\), as in Definition~\ref{deea}.
Clearly $B(\ba)\subseteq E(\ba)$. The next proposition supplies the
null-set property needed for the dimension staircase.

\begin{proposition}\label{prprpea}
Let $a$ be fixed. The bifurcation set $E(a)$ is a Lebesgue null set.
\end{proposition}
\begin{proof}
The negative doubling map preserves Lebesgue measure and is ergodic.  For
each $N\geq1$ put
\[
H_N:=\left(a,\frac12+\frac{1}{2(N+1)}\right).
\]
This is a fixed hole of positive Lebesgue measure.  Ergodicity implies that
almost every orbit visits $H_N$ (in fact, infinitely often), and therefore
\[
\operatorname{Leb}\bigl(S_{-2}(H_N)\bigr)=0
\qquad\text{for every }N\geq1.
\]
The completed endpoint conventions affect only the countable precritical
set and hence do not change this measure statement.

The geometric self-survival characterization proved independently in
the first part of Proposition~\ref{prop:edlm-null}, applied to $T_{-2}$,
gives $b\in S_{-2}(a,b)$ for every $b\in E(a)$.
That characterization does not use the present null-set conclusion.
If $b>1/2$, choose $N$ so large that
$\frac12+\frac{1}{2(N+1)}\le b$. Then $H_N\subseteq(a,b)$ and
$b\in S_{-2}(H_N)$. Therefore
\[
E(a)\subseteq\{1/2\}\cup\bigcup_{N=1}^{\infty}S_{-2}(H_N).
\]
Each set on the right has Lebesgue measure zero. Countable
subadditivity therefore gives $\operatorname{Leb}(E(a))=0$.
\end{proof}

\vspace{1em}
{\bf \noindent Proof of Theorem \ref{theorem2}}
\par
Fix $0<a\leq1/2$ and let $b$ range over $[1/2,1)$, with all
parameter neighbourhoods understood relative to this interval.
By Proposition~\ref{prprpea}, the
survivor-set bifurcation set $E(a)$ is Lebesgue null. Let $B(a)$ denote
the corresponding geometric entropy-bifurcation set. Since
$B(a)\subseteq E(a)$, it is also Lebesgue null. By definition,
$\lambda(a,\cdot):b\mapsto h_{\mathrm{top}}(S_{-2}(a,b))$ is locally
constant outside $B(a)$. It is nonincreasing by inclusion of survivor
sets and continuous by Proposition~\ref{continuity}, applied to
$T_{-2}$. At $b=1/2$, this continuity is understood from the right.

The dimension formula and the endpoint discussion in
Section~\ref{subsecdimensiondef} apply to every nonempty hole in the
present range.
If \(a<1/2<b\), the survivor set avoids
the discontinuity; if \(a=1/2<b\), the left-completed restriction is
continuous on the survivor set; if \(a<1/2=b\), the same holds for the
right-completed restriction. Since \(T_{-2}\) has constant slope of
absolute value \(2\), the dimension is the survivor entropy divided by
\(\log2\) in all these cases.
If \(a=b=1/2\), the hole is empty and
\(S_{-2}(1/2,1/2)=[0,1]\), whose Hausdorff dimension is \(1\).
By Remark~\ref{re1}, the corresponding symbolic space is the full
shift, with entropy \(\log2\), so the same identity holds in this case.
Thus, for every \(b\in[1/2,1)\),
\[
\dim_{\mathrm H}S_{-2}(a,b)
 =\frac{h_{\mathrm{top}}(S_{-2}(a,b))}{\log 2}.
\]
In particular, entropy continuity at the left endpoint gives
\[
\lim_{b\downarrow1/2}\eta(a,b)=\eta(a,1/2),
\]
and this limit equals \(1\) when \(a=1/2\).
Consequently, the Hausdorff dimension function
$\eta(a,\cdot)$ is continuous and nonincreasing on \([1/2,1)\), and
locally constant outside the Lebesgue-null set \(B(a)\). It is therefore
a devil's staircase on \([1/2,1)\). $\hfill\square$

\section{\texorpdfstring{Applications in $f\in\mathrm{EDLM}$}{Applications in EDLM}}\label{secfinal}

The symbolic classification above gives a common subshift for computing
the survivor entropy of an expansive decreasing Lorenz map with a hole.
For holes strictly containing the critical point, the two interval
survivor systems have the same itinerary subshift. Nonempty
critical-endpoint holes may contribute countable differences of
itineraries, while the empty-hole case uses the full kneading space.
The resulting entropy correspondence does not by itself imply equality
of geometric survivor sets or Hausdorff dimensions; a dimension
comparison requires additional geometric regularity of the address maps.

Analogously to the full case for $T_{-2}$ described in
Remark~\ref{re1}, we have the following conclusion for a general EDLM.

\begin{remark}\label{re-edlm-empty-hole}
Let $f\in\mathrm{EDLM}$ have critical point $c$ and kneading
invariants $(k_-,k_+)$. In the degenerate
case $a=b=c$, the hole $H=(c,c)$ is empty. No interval is removed and
$
S_f(c,c)=[0,1].
$
By Theorem~\ref{thm2}, the corresponding symbolic survivor space is
$\Omega(f)=\Omega(k_-,k_+)$, with the usual one-sided sequences retained.
Thus no additional symbolic restriction is imposed, and
$h_{\mathrm{top}}(S_f(c,c))=h_{\mathrm{top}}(f)$.
\end{remark}

\begin{proposition}\label{proptwoholes}
Let \(f\in\mathrm{EDLM}\) have kneading invariants \((k_-,k_+)\), and let
\(H=(a,b)\) be a hole with \(0<a\le c\le b<1\). Put
\[
\ba=\tau_f(a-),\qquad \bb=\tau_f(b+).
\]
Using the interval address map \(\pi_{-2}\), define
\[
H_0=(p_0,q_0),\qquad
p_0=\pi_{-2}(k_-),\quad q_0=\pi_{-2}(k_+),
\]
and
\[
H_1=(p_1,q_1),\qquad
p_1=\pi_{-2}(\ba),\quad q_1=\pi_{-2}(\bb).
\]
Then \(H_0\subseteq H_1\) and \(0<p_1\le1/2\le q_1<1\).
In all cases,
\[
h_{\mathrm{top}}(S_f(a,b))
=h_{\mathrm{top}}(\Omega(\ba,\bb))
=h_{\mathrm{top}}(S_{-2}(p_1,q_1)).
\]
For \(a<b\), itinerary coding conjugates the completed systems on
\(S_f(a,b)\) and \(S_{-2}(p_1,q_1)\) to compact subshifts having at
most countable symmetric difference with \(\Omega(\ba,\bb)\). If \(a<c<b\), both
subshifts equal \(\Omega(\ba,\bb)\).
\end{proposition}

\begin{proof}
Put \(\ell_0=1(12)^\infty\) and \(r_0=2(21)^\infty\), the two
critical itineraries of \(T_{-2}\). Strict monotonicity of the
one-sided itineraries at distinct points of an EDLM gives
{
\[
\ba\preceq k_-\preceq\ell_0\prec r_0\preceq k_+\preceq\bb.
\]}
{
The first inequality is an equality exactly when \(a=c\), and the last
one is an equality exactly when \(b=c\); these equalities refer to
the one-sided boundary sequences.}
The address map is nondecreasing in the alternating order, so
\(H_0\subseteq H_1\). Moreover,
{
\((12)^\infty\prec\ba\) and \(\bb\prec(21)^\infty\), since
\(0<a\) and \(b<1\). Since the only sequences of \(1/2\) are
\(\ell_0,r_0\), and the extremal sequences represent \(0,1\), this gives
\(0<p_1\le1/2\le q_1<1\). An equality at \(1/2\) can occur only
at the corresponding critical endpoint.}
In particular, \(H_0\cup H_1=H_1\); the two-hole notation records
the kneading restriction and the prescribed-hole restriction separately.

Suppose first that \(a<b\). Then \(p_1<q_1\): otherwise both
addresses would be \(1/2\), forcing \(\ba=\ell_0\), \(\bb=r_0\),
and hence \(a=b=c\) by the preceding inequalities.
When \(a=c<b\), the boundary sequence is \(\ba=\tau_f(c-)=k_-\),
and the completed dynamics uses \(f(c-)\) at each visit to \(c\).
When \(a<c=b\), it uses \(f(c+)\), while
\(\bb=\tau_f(c+)=k_+\). These one-sided boundary sequences need not
be the actual completed itineraries of \(c\).
For each of the two nonempty holes, the completed map is continuous
on the complement of the hole, since any retained critical endpoint
is approached only from its selected side. The finite-time survivor
sets are therefore compact by induction, and so is their intersection.
Label the retained left component by \(1\) and the right component
by \(2\), including the critical endpoint in its retained component.
The resulting itinerary map is continuous on the survivor set and
commutes with the dynamics. It is injective by topological expansivity:
two distinct points with the same itinerary would leave the interval
between them free of every first preimage of the critical point.
Thus each itinerary map is a conjugacy onto its image, a compact
subshift. Denote these images by \(X_f\) and \(X_{-2}\), respectively.

Put \(\Omega=\Omega(\ba,\bb)\). The kneading bounds imply
\(\Omega\subseteq\Omega(k_-,k_+)=\Omega(f)\), by
Theorem~\ref{thm2}. For points whose orbits avoid \(c\), order
preservation of itineraries makes avoidance of \((a,b)\) equivalent
to avoidance of \((\ba,\bb)\) by every shift. Hence \(X_f\) and
\(\Omega\) can differ only at itineraries of critical or precritical
points. There are countably many such points; each has one completed
itinerary and at most two one-sided itineraries.
The same argument applies to \(T_{-2}\), using \(\pi_{-2}\): away
from the preimages of \(1/2\), the address is unique, intertwines the
dynamics, and identifies the two hole-avoidance conditions. Consequently,
\[
X_f\triangle\Omega\quad\text{and}\quad
X_{-2}\triangle\Omega
\quad\text{are at most countable}.
\]
All three spaces are compact subshifts, so
Lemma~\ref{lem:countable-diff-entropy} gives equality of their entropies.
The itinerary conjugacies give the asserted equality of survivor
entropies.

If \(a<c<b\), the critical point and all its preimages escape in
each interval system. Moreover, \(k_-,k_+,\ell_0,r_0\) all lie
strictly inside \((\ba,\bb)\), so their prefixed one-sided
itineraries are also excluded from \(\Omega\).
There are therefore no exceptional itineraries, and
\(X_f=X_{-2}=\Omega\). This proves the common conjugacy assertion
for holes strictly containing the critical point.

Finally, let \(a=b=c\). Then \((\ba,\bb)=(k_-,k_+)\) and
\(H_1=H_0\), which may be empty. The empty-hole convention of
Remark~\ref{re-edlm-empty-hole} gives
\(\Omega(f)=\Omega(k_-,k_+)=\Omega(\ba,\bb)\). By that remark and the symbolic
entropy formula in Section~\ref{sese2},
\[
h_{\mathrm{top}}(S_f(c,c))
=h_{\mathrm{top}}(f)
=h_{\mathrm{top}}(\Omega(k_-,k_+)).
\]
If \(H_0\) is nonempty, its completed negative-doubling survivor
system has a compact itinerary subshift. The preceding address
argument shows that this subshift differs from \(\Omega(k_-,k_+)\)
by at most countably many critical or precritical itineraries, so
Lemma~\ref{lem:countable-diff-entropy} gives the required entropy
equality. If \(H_0\) is empty, Remark~\ref{re1} gives the full shift
and entropy \(\log2\) on both sides.
This establishes the symbolic and entropy assertions in the empty-hole
case without identifying the interval topologically with its kneading
space.
\end{proof}

\begin{proof}[Proof of Corollary~\ref{coco2}]
{
For the forward assertion, the hypotheses already give
\(0<a\le c\le b<1\). When \(a=c\) or \(b=c\), the corresponding
one-sided completion is used, and when \(a=b=c\) the empty-hole
convention of Remark~\ref{re-edlm-empty-hole} applies.}
Proposition~\ref{proptwoholes} associates this hole with
\(\Omega(\ba,\bb)\) and gives equality of their entropies, including
the critical-endpoint and empty-hole cases.
Weak admissibility and \(\sigma(\ba)\prec\sigma(\bb)\) then allow
Lemma~\ref{lem:dlm-realization} to realize this space as \(\Omega(g)\)
for a map \(g\in\mathrm{DLM}\).

For the converse, either \(0<a<1/2<b<1\), or
\(a=b=1/2\) and
\((k_-,k_+)=(1(12)^\infty,2(21)^\infty)\).
The geometric assertion of Theorem~\ref{thm:dlm-hole} therefore gives
\[
\tau_{-2}(a-)=k_-,\qquad \tau_{-2}(b+)=k_+,
\]
and
\[
\Omega\bigl(\tau_{-2}(a-),\tau_{-2}(b+)\bigr)
=\Omega(k_-,k_+)=\Omega(g).
\]
Taking $f=T_{-2}\in\mathrm{EDLM}$ proves the required representation;
in the degenerate case the hole $(1/2,1/2)$ is empty.
\end{proof}

The entropy correspondence does not imply equality of geometric survivor
sets or Hausdorff dimensions: these depend on the address maps.
A dimension comparison therefore requires additional geometric regularity,
such as suitable distortion estimates.

For boundary sequences arising from a hole of an EDLM, the following
remark records how the adjustment procedure leads to a DLM realization.

\begin{remark}
Let $f\in\mathrm{EDLM}$ have a hole $H=(a,b)$ with
{$0<a\le c\le b<1$}, and write
$
\ba=\tau_f(a-),$ $\bb=\tau_f(b+).
$
Suppose the adjustment procedure of Lemma~\ref{lelele11} and
Lemma~\ref{newlemma1} produces a weak-admissible pair
$(\widetilde{\ba},\widetilde{\bb})$. Assume in addition that
$
\sigma(\widetilde{\ba})\prec\sigma(\widetilde{\bb}),
$
 then
Lemma~\ref{lem:dlm-realization} gives a map $g\in\mathrm{DLM}$ with
\[
\Omega(g)=\Omega(\widetilde{\ba},\widetilde{\bb}).
\]
By the adjustment lemmas, $\Omega(g)$ and $\Omega(\ba,\bb)$ differ by
at most a countable set. Thus
Proposition~\ref{proptwoholes} and Lemma~\ref{lem:countable-diff-entropy}
give $h_{\mathrm{top}}(g)=h_{\mathrm{top}}(S_f(a,b))$, including the
critical-endpoint and empty-hole cases.
If no adjustment is needed, the realization is exact:
$\Omega(g)=\Omega(\ba,\bb)$.
\end{remark}

The parameter correspondence can be written explicitly. For fixed
\(f\in\mathrm{EDLM}\) and a fixed left endpoint
{\(0<a\le c\)}, define
\[
\Phi_f:{[c,1)}\longrightarrow{[1/2,1)},\qquad
\Phi_f(x):=\pi_{-2}\bigl(\tau_f(x+)\bigr),
\]
{with \(\tau_f(c+)=k_+\) at the left endpoint \(x=c\),}
and set
\[
m:=\pi_{-2}\bigl(\tau_f(a-)\bigr).
\]
Proposition~\ref{proptwoholes} gives
{\(0<m\le1/2\le\Phi_f(b)<1\)}. Its entropy statement gives, throughout
this range,
\begin{equation}\label{eq:edlm-entropy-correspondence}
h_{\mathrm{top}}\bigl(S_f(a,b)\bigr)
 =
h_{\mathrm{top}}\bigl(S_{-2}(m,\Phi_f(b))\bigr).
\end{equation}
The map $\Phi_f$ need not preserve null sets under pullback. Transferring
the parameter-measure conclusion therefore requires condition~(P) below.

\begin{proposition}\label{prop:edlm-null}
Fix \(f\in\mathrm{EDLM}\) and {\(0<a\le c\)}. Let \(E_f(a)\)
denote the set of {\(b\in[c,1)\) at which the map
\(t\mapsto S_f(a,t)\) is not locally constant relative to \([c,1)\)}.
Define \(E_{-2}(m)\) analogously as the bifurcation set of
\(T_{-2}\) {for right endpoints in \([1/2,1)\)}, with fixed
left endpoint $m:=\pi_{-2}\bigl(\tau_f(a-)\bigr)$. Then
\[
\begin{aligned}
E_f(a)&=\{b\in[c,1):b\in S_f(a,b)\}\\
&\subseteq C(f)\cup\Phi_f^{-1}\bigl(E_{-2}(m)\bigr).
\end{aligned}
\]
Here \(C(f)\) is the countable set consisting of the critical point
and its preimages, as in Definition~\ref{defofdecreasing}.
In particular, if the pullback condition
\begin{equation*}
\operatorname{Leb}\!\left(
\Phi_f^{-1}\bigl(E_{-2}(m)\bigr)\right)=0
\tag{P}\label{eq:pullback-null}
\end{equation*}
holds, then \(\operatorname{Leb}(E_f(a))=0\).
\end{proposition}

\begin{proof}
{
We first prove the characterization relative to the parameter interval
\([c,1)\). Suppose \(b\in S_f(a,b)\). If \(a<b\), then for every
\(t>b\) sufficiently close to \(b\), the point \(b\) belongs to the
hole \((a,t)\), and hence \(b\notin S_f(a,t)\). Thus \(b\in E_f(a)\).
In particular, when \(b=c>a\) and \(c\in S_f(a,c)\), the same argument
applies with the right-completed convention at \(c\). If \(a=b=c\),
Remark~\ref{re-edlm-empty-hole} gives
\(S_f(c,c)=[0,1]\), whereas \(S_f(c,t)\) omits the nonempty interval
\((c,t)\) for every \(t>c\); hence \(c\in E_f(c)\).

Conversely, suppose \(b\notin S_f(a,b)\). We first record a
finite-escape observation. For a nonempty hole \(J=(u,v)\) with
\(u\le c\le v\), put \(D_J=I\setminus J\), using the selected
one-sided completion if \(c\) is an endpoint of \(J\).
The resulting restriction \(f:D_J\to I\) is continuous: either
\(c\notin D_J\), or \(D_J\) approaches \(c\) only from the selected
side. Define the finite-time survivor sets by
\[
K_0(J)=D_J,\qquad
K_{n+1}(J)=\{x\in D_J:f(x)\in K_n(J)\}.
\]
These sets are compact by induction, and \(K_n(J)\) consists exactly
of the points avoiding \(J\) for times \(0,\ldots,n\).
Consequently, if a point escapes by time \(r\), a neighbourhood of
that point is disjoint from \(K_r(J)\). Nearby points therefore also
escape by time \(r\), although their first escape times may differ.

First let \(b>c\), and let \(r\ge1\) be the first time at which
the completed orbit of \(b\) enters \((a,b)\). Choose
\[
\max\{c,f^r(b)\}<d<b,\qquad J=(a,d).
\]
The orbit of \(b\) also enters \(J\) at time \(r\), so
\(b\notin K_r(J)\). Choose \(\delta>0\) such that
\[
U=(b-\delta,b+\delta)\subset(d,1)\setminus K_r(J),
\]
and set \(V=(b-\delta/2,b+\delta/2)\).
This argument also covers \(a=c\) when the orbit reaches \(c\)
before escaping: a nearby orbit crossing to the right of \(c\)
enters \(J\), while one remaining on the left follows the continuous
left branch. No continuity of the full iterate \(f^r\) across these
two possibilities is required.
For every \(t\in V\), we have \(J\subset(a,b)\cap(a,t)\) and
\[
(a,b)\triangle(a,t)
\subseteq[\min\{b,t\},\max\{b,t\}]\subset U.
\]
Every point of \(U\) escapes into the common hole \(J\) by time
\(r\). Thus every orbit surviving either \((a,b)\) or \((a,t)\)
avoids \(U\). The holes coincide outside \(U\), and their critical
endpoint conventions agree, so
\[
S_f(a,t)=S_f(a,b),\qquad t\in V.
\]
Thus \(b\notin E_f(a)\).

It remains to consider \(b=c\). The case \(c\in S_f(a,c)\) was
handled in the first part of the proof. If \(c\notin S_f(a,c)\),
then \(a<c\). Put \(J=(a,c)\) and let \(r\ge1\) be the first
escape time of the right-completed orbit of \(c\). Since
\(c\notin K_r(J)\), choose \(\delta>0\) with \(c+\delta<1\)
such that \(U=[c,c+\delta)\) is disjoint from \(K_r(J)\).
Every point of \(U\) escapes into \(J\) by time \(r\) under the
right completion. For \(c<t<c+\delta/2\), its orbit either follows
this escape itinerary or first reaches \(c\), which already belongs
to the enlarged hole \((a,t)\). Thus survivors of either hole avoid
\(U\). Since the holes coincide outside \(U\),
\(S_f(a,t)=S_f(a,c)\) for \(t\in[c,c+\delta/2)\).
Hence \(c\notin E_f(a)\).
Therefore
\[
E_f(a)=\{b\in[c,1):b\in S_f(a,b)\}.
\]}

Now let \(b\in E_f(a)\setminus C(f)\). By the characterization
just proved, \(b\in S_f(a,b)\), and its orbit avoids \(c\).
Its itinerary \(\bb=\tau_f(b)=\tau_f(b+)\) is therefore unique and,
by order preservation along the surviving orbit, belongs to
\(\Omega(\ba,\bb)\). Put
$
\epsilon:=\Phi_f(b)=\pi_{-2}\bigl(\tau_f(b+)\bigr).
$
Monotonicity of the address map implies that the addresses of all
shifts of \(\bb\) avoid \((m,\epsilon)\).
No such shift can equal \(\ell_0=1(12)^\infty\) or
\(r_0=2(21)^\infty\): the kneading bounds
\(k_-\preceq\ell_0\prec r_0\preceq k_+\) would force it to equal
\(k_-\) or \(k_+\), respectively. By strict order preservation of
itineraries at distinct points, this would make the orbit of \(b\)
reach \(c\), contrary to \(b\notin C(f)\).
Thus the address orbit avoids \(1/2\), and the address map intertwines
all its iterates with \(T_{-2}\). Consequently,
\(\epsilon\in S_{-2}(m,\epsilon)\).
Applying the self-survival characterization just proved to
\(T_{-2}\), with left boundary \(m\), we obtain
$
\epsilon\in E_{-2}(m).
$
Hence
\[
E_f(a)\subseteq C(f)\cup\Phi_f^{-1}\bigl(E_{-2}(m)\bigr).
\]
The exceptional set is needed because, at a precritical parameter,
the upper sequence used by \(\Phi_f\) need not be the sequence selected
by the surviving critical-endpoint completion. Each first-preimage
level of \(c\) is finite, so \(C(f)\) is countable.
Condition~(P) therefore implies \(\operatorname{Leb}(E_f(a))=0\).
\end{proof}

\begin{proof}[Proof of Corollary~\ref{coco3}]
The symbolic entropy assertion is~\eqref{eq:edlm-entropy-correspondence}.
Under~\eqref{eq:pullback-null}, Proposition~\ref{prop:edlm-null} gives
$\operatorname{Leb}(E_f(a))=0$. Let $B_f(a)$ denote the set where the
survivor entropy is not locally constant. Then $B_f(a)\subseteq E_f(a)$,
since local constancy of the survivor set
implies local constancy of its entropy. The entropy function is
nonincreasing by inclusion and continuous by Proposition~\ref{continuity},
and is locally constant outside the null set $B_f(a)$. This proves
the devil's staircase property, including its possible constant case.
\end{proof}

We next verify condition~\eqref{eq:pullback-null} for a concrete
piecewise-linear EDLM family.

\begin{example}\label{ex:edlm-linear-pullback}
Fix $1<\beta_1,\beta_2\leq2$ and set $c=1/2$. Define
\[
f(x)=f_{\beta_1,\beta_2}(x):=
\begin{cases}
1-\beta_1x,&0\le x<c,\\
\beta_2(1-x),&c<x\le1.
\end{cases}
\]
The slopes are $-\beta_1$ and $-\beta_2$, and
$f(0)=1$, $f(1)=0$. The critical values are
\[
f(c-)=1-\frac{\beta_1}{2},\qquad
f(c+)=\frac{\beta_2}{2},
\]
which lie in $[0,1]$ and give an upward jump of size
$(\beta_1+\beta_2-2)/2>0$.
Put $\lambda=\min\{\beta_1,\beta_2\}>1$.
If a nonempty open interval $J$ avoided all preimages of $c$, every
iterate would be affine on $J$, with
\[
1\geq |f^n(J)|\geq\lambda^n|J|
\]
for all $n$, a contradiction. Thus the precritical set is dense and
$f\in\mathrm{EDLM}$.

Fix $0<a<c$, and put
\[
\ba=\tau_f(a-),\qquad
m=\pi_{-2}(\ba),\qquad
\Phi(b)=\pi_{-2}\bigl(\tau_f(b+)\bigr),\quad c\le b<1.
\]
Here $m=m_{\beta_1,\beta_2}$ and $\Phi=\Phi_{\beta_1,\beta_2}$.
We verify condition~\eqref{eq:pullback-null} directly.

First, every fixed hole $H=(a,b)$ with $a<c<b<1$ has a
Lebesgue-null survivor set. To prove this, put
$S=S_f(a,b)$ and $\delta=\min\{c-a,b-c\}>0$.
If $\operatorname{Leb}(S)>0$, choose a density point
$x\in S\cap(0,1)$ and an arbitrarily small open interval
$J\subset(0,1)$ centred at $x$.
Let $n$ be the first time that $f^n(J)$ contains $c$;
such a time exists  by topological expansivity.
Since $f^j(J)$ avoids $c$ for $0\leq j<n$, the map
$f^n$ is affine and injective on $J$. Its slope is
$(-1)^n\beta_1^{n_1}\beta_2^{n_2}$, where $n_i$ counts the visits
to branch $i$ before time $n$, and $n_1+n_2=n$.
Its image $I=f^n(J)$ contains both $c$ and
$f^n(x)\notin(a,b)$. If $f^n(x)\leq a$, then
$(a,c)\subset I$; otherwise $f^n(x)\geq b$ and
$(c,b)\subset I$. Hence $|I\cap H|\geq\delta$.
Every point of $J$ whose $n$th image lies in $H$ escapes.
Because $f^n$ is affine on $J$ and $|I|\leq1$,
\[
\frac{|J\setminus S|}{|J|}
\geq\frac{|I\cap H|}{|I|}
\geq\delta.
\]
This contradicts the density-point property as $J$ shrinks to $x$,
so $\operatorname{Leb}(S)=0$.

The single parameter $b=c$ has Lebesgue measure zero. Now let
$b\in\Phi^{-1}(E_{-2}(m))\setminus\{c\}$, and write
$\epsilon=\Phi(b)$ and $\bb=\tau_f(b+)$.
The self-survival characterization in
Proposition~\ref{prop:edlm-null}, applied to $T_{-2}$, gives
$\epsilon\in S_{-2}(m,\epsilon)$.
Here we use only that characterization, which does not assume (P).
Since this survivor orbit avoids $1/2$, its negative-doubling
itinerary is unique and equals $\bb$, as $\pi_{-2}(\bb)=\epsilon$.
By Proposition~\ref{proptwoholes}, $\bb\in\Omega(\ba,\bb)$
and there is a point $x\in S_f(a,b)$ with itinerary $\bb$.
Strict order preservation of one-sided itineraries at distinct
points forces $x=b$.

Finally, set $d_N=c+(1-c)/(N+1)$ for $N\geq1$.
For each such $b>c$, some $N$ satisfies $d_N<b$.
Since $(a,d_N)\subset(a,b)$, the preceding argument gives
\begin{equation}\label{eq:beta-bifurcation-transfer}
\Phi^{-1}\bigl(E_{-2}(m)\bigr)
\subseteq\{c\}\cup\bigcup_{N\geq1}S_f(a,d_N).
\end{equation}
The singleton and all the survivor sets on the right are null by the
preceding arguments.
Consequently,
\[
\operatorname{Leb}\!\left(
\Phi_{\beta_1,\beta_2}^{-1}
\bigl(E_{-2}(m_{\beta_1,\beta_2})\bigr)\right)=0,
\]
which is condition~\eqref{eq:pullback-null}.
Thus Corollary~\ref{coco3} applies for every $1<\beta_1,\beta_2\leq2$
and every $0<a<1/2$.
\end{example}

The assertions with a fixed right endpoint follow by reflection. Put
$R(x)=1-x$ and, for $f\in\mathrm{EDLM}$, let
$\widehat f=R\circ f\circ R$. Then
$\widehat f\in\mathrm{EDLM}$ has critical point $1-c$, and
\[
R\bigl(S_f(a,b)\bigr)=S_{\widehat f}(1-b,1-a),
\]
where reflection exchanges the left and right completions. It preserves
entropy and Hausdorff dimension. The two branch formulas of $T_{-2}$
are invariant under this reflection, so Theorem~\ref{theorem2} gives the
dimension conclusion when $b$ is fixed and $a$ varies in $(0,1/2]$.
For a general EDLM, the analogous entropy conclusion follows by applying
Corollary~\ref{coco3} to $\widehat f$ with fixed left endpoint $1-b$;
the required pullback condition is condition~(P) for this reflected map.

\section*{Acknowledgement} Yiming Ding was supported by NSFC (No. 12271418); Yun Sun is supported
by the Fundamental Research Funds for the Central Universities (No. 3120626135) and NSFC (No. 12601157).

\section*{Conflict of Interest}
The authors declare that they have no known competing financial interests or personal relationships that could have appeared to influence the work reported in this paper.

\end{document}